\documentclass[leqno,11pt]{amsart}

\usepackage{geometry}

\usepackage[all]{xy} % diagrams
\usepackage{pgfplots}
\pgfplotsset{compat=1.18}

\usepackage{mathtools}
\usepackage{amsmath, amssymb, amsfonts, latexsym, mdwlist, amsthm}
\usepackage{subfig}
\usepackage{graphicx}
\usepackage{wrapfig}
\usepackage{comment}
\usepackage{tikz-cd}
\usepackage{faktor}

\usepackage[bookmarks, colorlinks, breaklinks, pdftitle={},
pdfauthor={Soheyla Feyzbakhsh}]{hyperref}
\hypersetup{linkcolor=blue,citecolor=blue,filecolor=black,urlcolor=blue}

\usepackage{tikz}
\usetikzlibrary{calc,trees,positioning,arrows,chains,shapes.geometric,%
    decorations.pathreplacing,decorations.pathmorphing,shapes,%
    matrix,shapes.symbols}

\tikzset{
>=stealth',
  punktchain/.style={
    rectangle,
    rounded corners,
    draw=black, thick,
    minimum height=3em,
    text centered,
    on chain},
  line/.style={draw, thick, <-},
  element/.style={
    tape,
    top color=white,
    bottom color=blue!50!black!60!,
    minimum width=8em,
    draw=blue!40!black!90, very thick,
    text width=10em,
    minimum height=3.5em,
    text centered,
    on chain},
  every join/.style={->, thick,shorten >=1pt},
  decoration={brace},
  tuborg/.style={decorate},
  tubnode/.style={midway, right=2pt},
}

\usepackage{paralist}
\setdefaultenum{(a)}{(i)}{}{}
\usepackage{enumitem} % for space-saving description environments}
\usepackage{graphicx}
\makeatletter
\newtheorem*{rep@theorem}{\rep@title}
\newcommand{\newreptheorem}[2]{%
\newenvironment{rep#1}[1]{%
 \def\rep@title{#2 \ref{##1}}%
 \begin{rep@theorem}}%
 {\end{rep@theorem}}}
\makeatother

\newcommand{\leftsup}[2]{{}^{#1}\!#2}

\newtheorem{Thm}{Theorem}[section]
\newreptheorem{Thm}{Theorem}
\newtheorem{Prop}[Thm]{Proposition}

\newtheorem{Lem}[Thm]{Lemma}

\newtheorem{Cor}[Thm]{Corollary}
\newreptheorem{Cor}{Corollary}

\newreptheorem{Con}{Conjecture}

\newtheorem{thm-int}{Theorem}

\theoremstyle{definition}
\newtheorem{Def-s}[Thm]{Definition}
\newtheorem{Def}[Thm]{Definition}

\newcommand{\ignore}[1]{}
\usepackage{amsmath,calligra,mathrsfs}

\newcommand{\sol}[1]{\textcolor{red}{#1}}
\newcommand{\Ale}[1]{\textcolor{purple}{#1}}

\renewcommand\_{^{}_}

\newcommand\PP{\mathbb P}
\newcommand\C{\mathbb C}

\newcommand\R{\mathbb R}

\newcommand\Z{\mathbb Z}
\newcommand\cA{\mathcal A}
\newcommand\cD{\mathcal D}
\newcommand\cH{\mathcal H}
\newcommand\cT{\mathcal T}

\newcommand\cV{\mathcal V}

\newcommand\du{\mathbb D}
\newcommand\ku{\operatorname{Ku}}

\newcommand\rd{\operatorname{\textbf{r}}}
\newcommand\dd{\operatorname{\textbf{d}}}
\newcommand\nd{\operatorname{\textbf{n}}}

\newcommand\T{{\mathcal T}_C}
\newcommand\cO{{\mathcal O}_C}
\newcommand\cK{{\mathcal O}_x}

\newcommand\Hom{\operatorname{Hom}}
\newcommand\Ext{\operatorname{Ext}}
\newcommand\Stab{\operatorname{Stab}}
\newcommand\gl{\operatorname{gl}}
\newcommand\Coh{\operatorname{Coh}}
\newcommand\cl{\operatorname{cl}}
\newcommand\cone{\operatorname{cone}}
\newcommand\Ho{\mathscr{H}\text{\kern -3pt {\calligra\large om}}}
\newcommand{\GL}{\widetilde{\textnormal{GL}}^{+}(2,\mathbb{R})}
\newcommand\stab{\operatorname{Stab}}
\newcommand\rk{\operatorname{rk}}

\begin{document}
	\title[Derived category of Coherent systems on curves]
	{Derived category of Coherent systems on curves and stability conditions\vspace{-2mm}}

\author{Soheyla Feyzbakhsh}
\address{Department of Mathematics, Imperial College, London SW7 2AZ, United Kingdom}
\email{s.feyzbakhsh@imperial.ac.uk}
    
\author{Aliaksandra Novik}
\address{Department of Mathematics, Imperial College, London SW7 2AZ, United Kingdom}
\email{a.novik24@imperial.ac.uk}

\begin{abstract} \noindent
%Let $C$ be a smooth projective curve of genus $g>0$. We describe an open subset of the space of Bridgeland stability conditions on the bounded derived category of coherent systems on $C$, and show that boundary of this locus is determined by the Brill--Noether theory of vector bundles on the curve $C$. 
Let $C$ be a smooth projective curve of genus $g>0$. We describe an open locus of Bridgeland stability conditions on the bounded derived category of coherent systems on $C$, and show that stability manifold detects the Brill--Noether theory of the curve. 
%unlike other curve-related categories whose stability spaces depend only on the genus.
%in the enlargement by weak stability conditions, its boundary is determined by the Brill--Noether theory of vector bundles on $C$.
\end{abstract}
%\vspace*{-12mm}
	\maketitle	

\vspace{-.4cm}

\setcounter{tocdepth}{1}
\tableofcontents

\section{Introduction}
Let $C$ be a smooth irreducible complex projective curve of genus $g > 0$. It was shown in \cite{macri:stability-conditions-on-curves} that the space of stability conditions (modulo the $\GL$–action) consists of a single point; thus there is no room to deform stability conditions and extract geometric information via wall–crossing on $\mathcal{D}^b(C)$. A first way around this is to embed $C$ into a “nice’’ higher–dimensional variety (e.g.\ a K3 surface), push the problem to the ambient variety, and apply wall–crossing there. This approach has resolved interesting questions in the Brill–Noether theory of curves, see e.g. \cite{bayer:brill-noether, feyz:mukai-program, feyz-li:clifford-indices, chunyi:stability-condition-quintic-threefold, bayer-li}. Its drawbacks are that it restricts attention to special curves admitting such embeddings and may lose track of certain data under pushforward to higher dimension. The alternative developed in this paper keeps the variety fixed but enlarges the category from coherent sheaves to \emph{coherent systems}. Following an idea originally suggested to us by Angela Ortega, we study the bounded derived category of coherent systems on $C$ and its Bridgeland stability conditions. As shown in \cite{AlexeevKuznetsov2025}, this category identifies with the Kuznetsov component of the blow-up of any smooth Fano threefold containing $C$ along $C$. For the origins of coherent systems and a review of the literature, see Section~\ref{related work}.

A (generalised) coherent system on $C$ is a triple $(V,E,\varphi)$, where $V$ is a finite-dimensional $\mathbb{C}$–vector space, $E$ is a coherent sheaf on $C$, and $\varphi\colon \cO\otimes V\to E$ is an arbitrary sheaf morphism.\footnote{In much of the literature on coherent systems one assumes $H^0(\varphi)$ is injective; in this paper we allow $\varphi$ to be arbitrary.}  We also write such a triple as
\[
[\cO \otimes V \xrightarrow{\ \varphi\ } E].
\]
The category of these triples, denoted $\mathcal{T}_C$, is abelian, and its bounded derived category is $\mathcal{D}(\mathcal{T}_C)$. The numerical Grothendieck group of $\cD(\cT_C)$ has rank three, identified via the class map
\[
\operatorname{cl}\colon \mathcal{N}\big(\cD(\cT_C)\big)\xrightarrow{\ \sim\ }\mathbb{Z}^3,
\]
which associates to any object $T=[\cO\otimes V\xrightarrow{\,\varphi\,} E]$ the vector
\[
\cl(T)
= \big(\rk E,\ \deg E,\ \dim_{\C} V\big)
\eqqcolon (\rd(T),\dd(T),\nd(T)).
\]
In the main part of the paper, we are concerned with an open subset
\[
\mathrm{Stab}^{\circ}\big(\mathcal{D}(\mathcal{T}_C)\big)\subset \mathrm{Stab}\big(\mathcal{D}(\mathcal{T}_C)\big)
\]
consisting of stability conditions $\sigma$ on $\mathcal{D}(\mathcal{T}_C)$ such that
\begin{enumerate}
  \item $[\cO \to 0]$ is $\sigma$–stable; and
  \item for each point $x\in C$, the object $[0 \to \cK]$ is $\sigma$–stable. 
\end{enumerate}
The complement of $\mathrm{Stab}^{\circ}\big(\mathcal{D}(\mathcal{T}_C)\big)$ will be addressed in subsequent work. There are two ways to construct stability conditions in $\Stab^{\circ}\big(\cD(\cT_C)\big)$.

\medskip
\noindent
$\bullet$ \textbf{(Gluing stability condition)} 
We know that $[\cO \to 0]$ is an exceptional object in $\cD(\cT_C)$, which induces a semiorthogonal decomposition 
\begin{equation}\label{eq.sod1}
    \cD(\cT_C) = \langle [\cO \to 0],\, \leftsup{\perp\,}{[\cO \to 0]} \rangle .
\end{equation}
The left orthogonal component $^{\perp}[\cO \to 0]$ is equivalent to the bounded derived category of coherent sheaves $\cD(C)$. 
By \cite[Theorem 2.7]{macri:stability-conditions-on-curves}, there is an isomorphism
\begin{align*}
      \GL &\xrightarrow{\;\cong\;} \Stab(\cD(C)), \qquad g \mapsto \sigma_g,
      %g = (T,f) &\longmapsto \sigma_{g} = \big(\Coh^{f(0)}(C),\, T^{-1} \!\circ Z_{\mu}\big), 
\end{align*}
where $\sigma_g$ is defined in~\eqref{eq_sigma_g} in Section~\ref{sec gluing}. Let $\sigma_{\mathcal{V}}$ denote the trivial stability condition on the bounded derived category of $\mathbb{C}$-vector spaces $\cD(\mathcal{V})$ (generated by $[\cO \to 0]$), given by %whose heart is 
\[
\cA_{\mathcal{V}} = \{\C^{\oplus n}\}_{n \ge 0}, \qquad Z_{\mathcal{V}}(n) = -n.
\]
In \cite{Collins-gluing-stability-conditions} it is shown that, for a suitable choice of $\sigma_{g}$ on $\cD(C)$, one can glue $\sigma_{\mathcal{V}}$ and $\sigma_{g}$ to obtain a pre-stability condition on the full category $\cD(\cT_C)$, denoted
\[
\gl^{(1)}(\sigma_{\mathcal{V}}, \sigma_{g}) \in \Stab^{\circ}\big(\cD(\cT_C)\big),
\]
where the superscript~$(1)$ refers to the first type of semiorthogonal decomposition given in \eqref{eq.sod1}.

\medskip
\noindent
$\bullet$ \textbf{(Tilting stability condition)} 
In parallel with Bridgeland’s construction of stability conditions on surfaces \cite{bridgeland:K3-surfaces}, we start with the slope function $\mu(T) = \frac{\dd(T)}{\rd(T)}$ (see Definition \ref{def-alpha}) for objects $T \in \cT_C$ with nonzero rank. This induces the notion of $\mu$–stability on $\cT_C$. Then every object in $\cT_C$ admits a Harder–Narasimhan filtration with respect to $\mu$–stability. 
We write $\mu^{+}(T)$ and $\mu^{-}(T)$ for the maximal and minimal slopes occurring in the HN filtration of $T$, respectively. 

Tilting the abelian category $\cT_C$ at slope $b \in \R$ yields the torsion pair $(\mathbb{T}^{b},\, \mathbb{F}^{b})$, where 
\begin{enumerate}
  \item[$-$] $\mathbb{T}^{b}$ is the full subcategory of objects $T \in \cT_C$ satisfying $\mu^{-}(T) > b$, and 
  \item[$-$] $\mathbb{F}^{b}$ is the full subcategory of objects $T \in \cT_C$ satisfying $\mu^{+}(T) \le b$.
\end{enumerate}

\noindent
To determine the region of tilting stability conditions, we define the \emph{Brill–Noether function}, analogous to the Le~Potier function \cite{DLP}, as
\begin{equation*}
    \Phi_{C} \colon \R \longrightarrow \R, \qquad
    \Phi_{C}(x) \coloneqq 
    \limsup_{\mu \to x}
    \left\{
      \frac{h^{0}(C,F)}{\rk(F)} \ \middle|\ 
      \begin{array}{l}
      F \in \Coh(C) \text{ semistable,}\\
      \mu(F) = \mu
      \end{array}
    \right\}.
\end{equation*}

Our main theorem states that these two types of stability conditions control the full open subset $\mathrm{Stab}^{\circ}\big(\mathcal{D}(\mathcal{T}_C)\big)$.  
\begin{Thm}[= Theorem~\ref{thm-main}]\label{thm-main-intro}
Up to the action of $\GL$, any stability condition $\sigma \in \mathrm{Stab}^{\circ}\big(\mathcal{D}(\mathcal{T}_C)\big)$ is of one of the following types:  
  \begin{enumerate}
      \item[\textbf{Type A.}] $\sigma$ is the gluing $\gl^{(1)}(\sigma_{\mathcal{V}}, \sigma_g)$ where $g= (T, f) \in \GL$ with $f(0) < \frac{1}{2}$, and $\sigma_{\mathcal{V}}$ is the stability condition on $\cD(\mathcal{V})$ with the heart $\cA_{\mathcal{V}}= \{\mathbb{C}^{\oplus n}\}_{n\geq 0}$ and stability function $Z_{\mathcal{V}}(n) = -n$.  
   
   \item[\textbf{Type B.}] The heart of $\sigma$ is given by $\cA(b) = \langle \mathbb{F}^b[1], \mathbb{T}^b \rangle$ for $b \in \mathbb{R}$ and the stability function is given by 
  \begin{equation*}
  Z_{b,w} \colon \mathcal{N}(\cD(\mathcal{T})) \rightarrow \mathbb{C} \ ,\;\;\; Z_{b,w}(T) = -\nd(T)+w\rd(T)+i (\dd(T)-b\rd(T))
  \end{equation*}
  where $w > \Phi_C(b)$. 
   \end{enumerate}
  \end{Thm}
Note that the condition $f(0)<\frac{1}{2}$ for Type~A stability conditions is necessary for the gluing construction. Moreover, the intersection of the Type~A and Type~B loci is described in Proposition~\ref{prop-gluing-type 1}.

The Type~B stability conditions form a real two–dimensional family parameterised by \((b,w)\) with \(w>\Phi_C(b)\), which we denote by
\(\sigma_{b,w}:=(\cA(b),Z_{b,w})\).
Their wall–chamber decomposition is described in Proposition~\ref{prop. locally finite set of walls}.
Whereas the stability manifold of \(\cD(C)\) is essentially independent of the curve,
the stability manifold of \(\cD(\cT_C)\) is closely controlled by the Brill–Noether theory of \(C\).
We will return to applications of wall-crossing in this two–dimensional slice to the Brill–Noether theory of vector bundles on curves in subsequent work.

As a result of Theorem \ref{thm-main-intro}, we can describe the complex manifold. 

\begin{Cor}[= Corollary~\ref{cor complex manifold}]\label{cor-stability-manifold-intro}
We have
\[
\Stab^{\circ}\big(\cD(\cT_C)\big)=U_A\cup U_B.
\]
as the union of the open loci \(U_A\) and \(U_B\), described as follows:
\begin{itemize}
  \item[\(\bullet\)] The open locus \(U_A\) consists of stability conditions \(\sigma\) such that
\([\cO\to0]\), \([0\to\cK]\), and \([0\to\cO]\) are
\(\sigma\)-stable of phases \(\phi_1,\phi_2,\phi_3\), respectively. Writing
\[
z_i=m_ie^{i\pi\phi_i},
\qquad
m_i>0,\ \phi_i\in\R,
\]
we obtain
\begin{equation}\label{eq.UA}
    U_A=
\Bigl\{
(z_1,z_2,z_3)\in(\widetilde{\C^*})^3
\ \Big|\
\phi_1-1<\phi_3<\phi_2<\phi_3+1
\Bigr\} \footnote{Via Bridgeland's local homeomorphism, \(U_A\) is identified with an open subset of
\((\widetilde{\C^*})^3\), where
\(\widetilde{\C^*}\cong\R_{>0}\times\R\) denotes the universal covering of
\(\C^*\).}.
\end{equation}

  %The open locus \(U_A\) consists of stability conditions \(\sigma\) such that \([\cO\to 0],\ [0\to \cK],\ [0\to\cO]\) are $\sigma$-stable of phases \(\phi_1,\phi_2,\phi_3\), respectively. Then, writing \(z_i=m_i e^{i\pi\phi_i}\) with \(m_i>0\) and \(\phi_i\in\R\), we have
  %\begin{equation}\label{eq.UA}
  %    U_A=\Big\{(z_1,z_2,z_3)\in(\C^*)^3 \ \Big|\ 
  %\phi_1-1<\phi_3<\phi_2<\phi_3+1 \Big\}.
  %\end{equation}

  \item[\(\bullet\)] The open locus $U_B$ consists of stability conditions \(\sigma\) such that \([\cO\to 0]\) and \([0\to \cK]\) are \(\sigma\)-stable and
  \(\phi_{\sigma}([0\to \cK])<\phi_{\sigma}([\cO\to 0])\).
  On \(U_B\) the right action of \(\GL\) is free, and
  \[
  U_B\big/\GL\;=\;\{\,b+iw\in\C \mid w>\Phi_C(b)\,\}.
  \]
\end{itemize}
\end{Cor}

Note that, in the Corollary  above, \(U_A\) (resp. \(U_B\)) consists, up to the \(\GL\)-action, of the Type~A (resp. Type~B) stability conditions of Theorem~\ref{thm-main-intro}, and \(U_A\cap U_B\neq \emptyset\).

On the other hand, analogously to the space of (weak) stability conditions on varieties of dimension $\ge 2$, the large-volume limit is also significant in our two–dimensional slice $(b,w)$: in this regime, we recover the classical notion of $\alpha$–stability for coherent systems, see Proposition~\ref{prop-positive alpha} for more details. 

\begin{comment}
    In the final section we study another open locus of the stability manifold, obtained from a second
type of gluing stability condition induced by
the semiorthogonal decomposition
\begin{equation*} \cD(\cT_C) = \langle {[\cO \to \cO]}^{\perp},\ [\cO \to \cO]\rangle , 
\end{equation*} 
where $[\cO \to \cO]^{\perp}$ is equivalent to $\cD(C)$. We denote by $\gl^{(2)}(\sigma_g, \sigma_{\mathcal{V}})$ the gluing \Ale{of} stability condition $\sigma_g$ on $\cD(C)$ to the trivial one $\sigma_{\mathcal{V}}$ on $\cD(\mathcal{V})$. Our final theorem states that if $[0\to\cO]$ and $[0\to\cK]$ are $\sigma$-stable for all skyscraper sheaves $\cK$ of points $x\in C$, then $[0\!\to\!E]$ is $\sigma$-stable if and only if $E$ is a slope-stable sheaf on $C$ (up to a shift).
\end{comment}

In the final section, we describe another open locus $\widetilde{\Stab}^{\circ}(\mathcal{D}(\T))$ of the stability manifold, consisting of stability conditions $\sigma$ such that $[0\to\cO]$ and $[0\to\cK]$ are $\sigma$-stable for all $x\in C$. To this end, we study stability conditions arising from the gluing construction with respect to the second semiorthogonal decomposition
\[
\cD(\cT_C)=\langle {[\cO \to \cO]}^{\perp},\ [\cO \to \cO]\rangle,
\]
where $[\cO \to \cO]^{\perp}$ is equivalent to $\cD(C)$. We denote by $\gl^{(2)}(\sigma_g,\sigma_{\mathcal V})$ the stability condition obtained by gluing a stability condition $\sigma_g$ on $\cD(C)$ with the trivial stability condition $\sigma_{\mathcal V}$ on $\cD(\mathcal V)$. Our final theorem states that if $\sigma\in\widetilde{\Stab}^{\circ}(\mathcal{D}(\T))$, then $[0\to E]$ is $\sigma$-stable if and only if $E$ is a slope-stable sheaf on $C$, up to a shift.

\begin{Thm}[= Theorem~\ref{thm-second open subset}] \label{theorem-intro-second gluing}
If $\sigma \in \widetilde{\Stab}^{\circ}(\cD(\T))$, then, up to the $\GL$-action, $\sigma$ is either of the form $\gl^{(1)}(\sigma_g, \sigma_{\mathcal{V}})$ or $\gl^{(2)}( \sigma_g, \sigma_{\mathcal{V}})$ for some $g \in \GL$. 
\end{Thm}

\subsection{Foundational and Related Works}\label{related work}
The notion of a coherent system on a smooth projective curve $C$—a pair $(E,V)$ with $E$ a vector bundle and $V\subset H^0(C,E)$ a linear subspace of dimension $n$—together with the concept of $\alpha$-(semi)stability depending on a real parameter $\alpha$, originates in Le~Potier’s monograph~\cite{potier-1}. 
These ideas were foreshadowed by Bradlow’s study of ``stable pairs'' (the case $n=1$)~\cite{Bradlow-matrics} and further developed in the moduli-theoretic analyses of Thaddeus \cite{Thaddeus-stable-pairs} and He %, the latter in a general projective setting that specializes to curves
\cite{He1996}.

One of the first systematic treatments of coherent systems on curves of arbitrary type \((r,d,n)\) is due to Bradlow–García-Prada–Muñoz–Newstead~\cite{BradlowGaraPradaMunozNewstead2003}. They constructed projective moduli spaces of \(\alpha\)-stable coherent systems, identified the discrete set of critical values of \(\alpha\), and related the large-\(\alpha\) chamber to the classical Brill–Noether loci. 
A substantial body of subsequent work has investigated the birational and topological geometry of these moduli spaces (see, e.g.,~\cite{cris-hodge,bradlow-coherent,bradlow-biratonal}) and their non-emptiness (see, e.g.,~\cite{news-exit,news-det,news-small,bigas-exist,Zhang-ex}); for an overview, see Newstead’s survey~\cite{news-survey}.

A detailed analysis of coherent systems and their moduli spaces has been carried out for various classes of special curves~\cite{angela-special}, including the projective line~\cite{news-p}, elliptic curves~\cite{lange-elliptic}, and Petri-general curves~\cite{news-petri}.  
The theory has also been extended to singular settings such as nodal or cuspidal curves of compact type~\cite{usha-nodal, ballico-cuspidal}.

Beyond questions of existence and birational geometry, coherent systems play a key role in the study of Butler’s conjecture on the stability of kernels of evaluation maps~\cite{ortega-butler, leticia-butler}, which will be discussed in detail in subsequent work.

In very recent developments, Kuznetsov and Alexeev have shown that derived categories of coherent systems naturally arise in the context of compact-type degenerations of curves~\cite{AlexeevKuznetsov2025}. From the point of view of stability conditions, the space of stability conditions on the bounded derived category of holomorphic triples was studied in \cite{Eva:bridgeland-stability-condition-on-holomorphic-triples}, where objects are triples $(E_1,E_2,\varphi)$ with $E_1,E_2$ coherent sheaves on a curve $C$ and $\varphi$ an arbitrary morphism. Unlike our case, that stability manifold depends only on the genus of the curve $C$ and not on the ambient geometry of $C$. Moreover, in the very recent preprint~\cite{comma}, stability conditions on abelian comma categories—of which the category of coherent systems is an example—are studied. We have been informed of work in progress~\cite{new} in which the authors construct stability conditions on the bounded derived category of coherent systems on integral curves via tilting.

\subsection{Organization of the paper}
In Section~\ref{sec D(T)}, we introduce generalised coherent systems and analyse their derived category. Section~\ref{section tilt} establishes the existence of a real two-dimensional slice of stability conditions arising from the tilting construction. In Section~\ref{sec gluing}, we review the technique of gluing stability conditions with respect to a semiorthogonal decomposition and demonstrate their existence in our setting. In Section~\ref{sec main thm}, we study the open locus of the stability manifold and prove Theorem~\ref{thm-main-intro}. In Section~\ref{sec chamber dec}, we first describe the wall-and-chamber decomposition within the two-dimensional slice, and then study the large volume limit, recovering classical stability of coherent systems. In Section~\ref{sec second type}, we study the second open locus of the stability manifold and prove Theorem~\ref{theorem-intro-second gluing}. Finally, in the appendix, we specialize to the case where $C$ is an elliptic curve and give a complete description of the locus of geometric stability conditions.

\subsection*{Acknowledgments}
We are especially grateful to Angela Ortega for drawing our attention to the category of coherent systems, and to Sasha Kuznetsov for suggesting the idea behind Lemma~\ref{lem-dual} and for his helpful comments on a preliminary version of this paper. We also thank Arend Bayer, Gavril Farkas, Richard Thomas, and Yukinobu Toda for helpful discussions.
S.F.\ acknowledges support from the Royal Society (URF/R1/23119). 
     
\section{Derived category of coherent systems}\label{sec D(T)}
	Let $C$ be a smooth irreducible complex projective curve of genus $g$. Let $\mathcal{V}$ be the abelian category of finite-dimensional $\mathbb{C}$-vector spaces, and let $\mathcal{T}_C$ be the category of triples $(V, E , \varphi)$ where $V \in \mathcal{V}$, $E \in \text{Coh}(C)$, and $\varphi \colon \cO \otimes V \to E$ is a sheaf morphism. A morphism $\psi \colon (V, E , \varphi) \to (V', E' , \varphi')$ between two triples consists of a pair $\psi = (\psi_1, \psi_2)$ of a morphism of vector spaces $\psi_1 \colon V \to V'$ and a sheaf morphism $\psi_2 \colon E \to E'$ such that the following diagram commutes:%so that we have the following commutative diagram 
 \begin{equation*}
     \xymatrix{
	\
    \cO \otimes V\ar[r]^{\text{id}\otimes \psi_1}\ar[d]^{\varphi}&\cO \otimes V'\ar[d]^{\varphi'} \\
			 E\ar[r]^{\psi_2}&E'.	}
 \end{equation*}
 We usually denote the triple $(V, E , \varphi)$ by $[\cO \otimes V \xrightarrow{\varphi} E]$. One can easily check that $\mathcal{T}_C$ is an abelian category. Note that $\mathcal{T}_C$ contains the non-abelian category of coherent systems $[\mathcal{O}_C \otimes V\xrightarrow{\varphi} E]$ where $H^0(\varphi)$ is injective. We denote by $\cD(\mathcal{T}_C)$ the bounded derived category of $\mathcal{T}_C$. Its objects are the same as the objects of the category of complexes Kom$(\mathcal{T}_C)$ namely, complexes of the form 
	\begin{equation*}\label{objects}
	\xymatrix{
		\dots \ar[r] &\mathcal{O}_C \otimes V_{{i-1}}\ar[r]\ar[d]&\mathcal{O}_C \otimes V_{i} \ar[r]\ar[d]&\mathcal{O}_C \otimes V_{i+1}\ar[d]\ar[r]&\dots \\
			\dots \ar[r] &E_{i-1}\ar[r]&E_i\ar[r]&E_{i+1}\ar[r]&\dots	}
	\end{equation*} 
%Thus any object in $\mathcal{D}(\mathcal{T}_C)$ can be written as $[\cO \otimes V^{\bullet} \xrightarrow{f} E^{\bullet}]$ where $V^{\bullet}$ is a complex of vector bundles, $E^{\bullet}$ is a complex of coherent sheaves and $f$ is morphism in Kom$(\Coh(C))$.

 We may enlarge the category $\mathcal{T}_C$ to $\mathcal{T}_C^{\text{quasi}}$ which contains triples $(V, E, \varphi)$ such that $E$ is a quasi-coherent sheaf and $V$ is allowed to be an infinite-dimensional $\mathbb{C}$-vector space. By \cite[Theorem 1.3]{He:espaces-de-modules}, an object in $\mathcal{T}_C^{\text{quasi}}$ is injective if and only if it is of the form 
   \begin{equation*}
       [\cO \otimes V \rightarrow 0] \oplus [\cO \otimes \Hom(\cO, I) \xrightarrow{ev} I],
   \end{equation*}
   where $I$ is an injective quasi-coherent sheaf on $C$. 
   \begin{Lem}\label{lem-injective-resol}
	Any object $[\mathcal{O}_C \otimes V \xrightarrow{\varphi} E] \in \mathcal{T}_C^{\text{quasi}}$ has an injective resolution of the form
	   \begin{equation*}
 	\xymatrix{
0 \ar[r]\ar[d]&\mathcal{O}_C \otimes V \ar[r]^-{d_0'}\ar[d]_-{\varphi}& \mathcal{O}_C \otimes (V_1 \oplus H^0(I_1)) \ar[r]^{d_1'}\ar[d]_-{(0, \text{ev})}& \mathcal{O}_C \otimes (V_2 \oplus H^0(I_2))\ar[r]\ar[d]_-{(0, \text{ev})}&\mathcal{O}_C \otimes V_3\ar[d]\ar[r]&0\ar[d] \\
0 \ar[r]& E \ar[r]^-{d_0}&I_1 \ar[r]^{d_1}& I_2\ar[r]  & 0\ar[r]&0    
}
  \end{equation*} 
  	for suitable vector spaces $V_1, V_2$ and $V_3$ where $0 \rightarrow E \xrightarrow{d_0} I_1 \xrightarrow{d_1} I_2 \rightarrow 0$ is an injective resolution of $E$. % in the category $QCoh(C)$.
	\end{Lem}
\begin{proof}
We may write $V \cong \ker H^0(\varphi) \oplus V_1$, then injection $H^0(d_{0}) \colon H^0(E) \to H^0(I_1)$ gives the injection in $\mathcal{T}_C^{\text{quasi}}$: 
	   \begin{equation*}
 	\xymatrix@C=3.9em{
 0\ar[r]\ar[d] &\mathcal{O}_C \otimes (\ker H^0(\varphi) \oplus V_1)   \ar[r]^{\hspace*{-1em}(\mathrm{id}, H^0(d_0))~~~~}\ar[d]_-{\varphi}&\mathcal{O}_C \otimes ({\ker H^0(\varphi) \oplus H^0(I_1)}) \ar[d]_-{(0, \mathrm{ev})}\\
 0\ar[r]&E \ar[r]^{d_0}&I_1.     
}
  \end{equation*}
Then the quotient in $\mathcal{T}_C^{\text{quasi}}$ is of the form $[\cO \otimes V' \xrightarrow{\varphi'} I_2]$ where $V' = H^0(I_1)/V_1$. One may apply the same argument to construct $V_2$ and the map $d_2'$, and then $V_3$ will be the final quotient. 
\end{proof}

For an object $T \in \mathcal{T}_C^{\text{quasi}}$, since $\Hom(T, -)$ is a left exact covariant functor, and the category of $T^{\text{quasi}}$ has enough injectives, we can define $R\Hom(T, -)$ as the right derived functor of $\Hom(T, -)$. As a consequence of Lemma~\ref{lem-injective-resol}, the category $\T$ has homological dimension~$2$; that is, for any objects $T_1,T_2\in\T$, we have $\Ext^k(T_1,T_2)=0$ for all \(k\notin\{0,1,2\}\).

\begin{Prop}{\cite[Proposition 1.5]{He:espaces-de-modules}}, \cite[Lemma 2.6]{AlexeevKuznetsov2025}\label{prop.hom seqeunce}
%Take two objects 
Let $T_i =[\cO \otimes V_i \xrightarrow{\varphi_i} E_i] \in \mathcal{T}_C$ for $i=1, 2$. %We also have
Then, there is a long exact sequence of vector spaces 
\begin{align*}
    0 & \rightarrow \Hom(T_1, T_2) \rightarrow \Hom(V_1, V_2) \oplus \Hom(E_1, E_2) \rightarrow \Hom(\cO \otimes V_1 , E_2) \\
    & \rightarrow \Ext^1(T_1, T_2) \rightarrow \Ext^1(E_1, E_2) \rightarrow \Ext^1(\cO \otimes V_1 , E_2) \rightarrow \Ext^2(T_1, T_2) \rightarrow 0.
\end{align*} 
 %\item[(b)] If the evaluation map $\cO \otimes V_1 \rightarrow F$ is injective with cokernel  $\Theta$, then there exists a canonical homomorphism 
 %\begin{equation*}
  %   \Ext^i(\Theta, E_2) \rightarrow \Ext^i(E_1^{\bullet}, E_2^{\bullet})
 %\end{equation*}
 %which is an isomorphism for $i \geq 2$, and epimorphism for $i=1$. It is an isomorphism for $i=1$ if the canonical morphism $\Hom(E_1^{\bullet}, E_2^{\bullet}) \rightarrow \Hom(V_1, V_2)$ is surjective.  
%\end{enumerate}
\end{Prop}
%The above Proposition will be the key tool to control Ext groups, for intance we get simiply the vansiing 

For any $T_1, T_2\in \T$ we define
\begin{equation*}
\chi(T_1, T_2) \coloneqq \sum_{k}  (-1)^k\dim_{\mathbb{C}} \text{Hom}(T_1 , T_2[k]).
\end{equation*}
Take two objects $T_i = [\cO \otimes V_i \xrightarrow{\varphi _i} E_i] \in \mathcal{T}_C$ for $i=1, 2$ with $\cl(T_i) = (r_i, d_i, n_i)$. Then Proposition \ref{prop.hom seqeunce} implies that 
\begin{align*}
   \chi(T_1, T_2) = & \ \dim_{\mathbb{C}} \Hom(V_1, V_2) + \sum_{k=0}^1 \dim_{\mathbb{C}}\Hom(E_1, E_2[k]) - \sum_{k=0}^1 \dim_{\mathbb{C}}\Hom(\cO \otimes V_1, E_2[k]) \\
   = & \ n_1n_2 +\chi(E_1, E_2) - \chi(\cO \otimes V_1, E_2)\\
   = & \ n_1n_2 + r_1d_2-r_2d_1 + r_1r_2(1-g) - (n_1d_2+ n_1r_2(1-g)  )\\
   %= & \ r_1 (d_2 +r_2(1-g)) -n_1(d_2 + r_2(1-g)) +n_1n_2 -r_2d_1\\
    = & \ (d_2+r_2(1-g))(r_1-n_1) +n_1n_2 -r_2d_1\\
    = & \ r_1 (d_2 +r_2(1-g)) -d_1r_2 + n_1(n_2 -d_2-r_2(1-g)). 
\end{align*}
%\medskip

\subsection{Semiorthogonal decompositions} 
\subsection*{First type:}
Since $[\cO \rightarrow 0]$ is an injective simple object, it is an exceptional object, so we have an exact functor  
	\begin{equation*}
	i_* \colon \mathcal{V} \rightarrow \mathcal{T}_C \;\;\;\;\ \mathbb{C} \mapsto [\mathcal{O}_C \rightarrow 0].
	\end{equation*} 
 Then we take the corresponding derived functor and we obtain a fully faithfull embedding $i_* \colon \cD(\mathcal{V}) \rightarrow \cD(\mathcal{T}_C)$ which has both adjoints $i^*  \dashv i_* \dashv i^!$,
where $i^*, i^! \colon \cD(\mathcal{T}_C) \rightarrow \mathcal{D}(\mathcal{V})$ are defined as 
\begin{equation}\label{eq-adjoints to i}
i^*(T) =  R\Hom(T, [\cO \rightarrow 0])^*, \quad i^!(T) =  R\Hom([\cO \rightarrow 0], T).
\end{equation}

On the other hand, the exact functor $j_* \colon \text{Coh}(C) \rightarrow \mathcal{T}_C$ sending $E$ to $[0 \to E]$ induces the fully faithful embedding $j_* \colon \cD(C) \to \cD(\mathcal{T}_C)$. Since $\cD(C)$ is saturated, $j_*\cD(C)$ is an admissible subcategory of $\cD(\cT_C)$, i.e. it has left and right adjoints $ j^* \dashv j_* \dashv j^!$. 

\begin{Lem}\label{lem-sod1}
    There is a semiorthogonal decomposition
	 \begin{align}\label{sod.1}
	    \mathcal{D}(\mathcal{T}_C) =\ &  \langle i_*\cD(\mathcal{V}) ,\ j_*\mathcal{D}(C) \rangle. 
	 \end{align}
\end{Lem}
\begin{proof}
    Since $[\cO \to 0]$ is exceptional, we only need to show that $^{\perp}[\cO \to 0]\simeq j_*\mathcal{D}(C)$. Note  that since $[\cO \to 0]$ is injective, for any $T
    %=[\cO \otimes V^{\bullet} \to E^{\bullet}]
     \in \mathcal{D}(\T)$ we have
    \[
    \Hom_{\mathcal{D}(\T)}(T, [\cO \to 0])= \Hom_{\mathrm{Kom}(\T)}(T, [\cO \to 0])%=\Hom(V^{\bullet}, \mathbb{C})
    .
    \]
   This implies that for any $F\in \mathcal{D}(C)$, there is vanishing $\Hom(j_*F, [\cO \to 0])=0$, so $j_*\mathcal{D}(C)\subset~^{\perp}[\cO \to 0]$. And vice versa, if $T
   %=[\cO \otimes V^{\bullet} \to E^{\bullet}]
   \in\ ^{\perp}[\cO \to 0]$, then $i^*T =0$, %$V^{\bullet}=0$
   thus $T = j_*E^{\bullet}$ for some $E^{\bullet} \in \cD(C)$.  
\end{proof}
 Therefore, any object $T \in \cD(\mathcal{T}_C)$ lies in the exact triangle 
\begin{equation}\label{rep.1}
    R_{[\cO \rightarrow 0]}(T)  = j_*j^!T   \to T \to i_*i^{*}T \xrightarrow{\varphi_T}j_*j^!T [1]. 
\end{equation}

\begin{Lem}\label{lem-cohomology}
    For every $i \in \mathbb{Z}$, we have the short exact sequence in $\cT_C$
    \begin{equation}\label{eq.coho}
0 \to \cH^{i}(j_*j^!T) \to \cH^i(T) \to \cH^i(i_*i^{*}T) \to 0.
\end{equation}
\end{Lem}
\begin{proof}
    Taking cohomology with respect to the heart $\cT_C$ of the exact triangle~\eqref{rep.1} gives a long exact sequence of objects in $\cT_C$:
\[
\dots \to \cH^{i-1}(i_*i^{*}T) \xrightarrow{d_{i-1}} \cH^{i}(j_*j^!T)
\to \cH^i(T) \to \cH^i(i_*i^{*}T) \xrightarrow{d_i} \cH^{i+1}(j_*j^!T) \to \dots
\]
Since $\Hom([\cO\to 0], j_*F)=0$ for any sheaf $F\in\Coh(C)$, the morphisms $d_i$ vanish, and hence the claim follows. 
\end{proof}
\subsection*{Second type:} Applying Proposition \ref{prop.hom seqeunce}, one can easily check the object $[\cO \xrightarrow{\text{id}} \cO]$ is also an exceptional object inducing the exact embedding
	\begin{equation*}
	i'_* \colon \cD(\mathcal{V}) \rightarrow \cD(\mathcal{T}_C) \;\;\;\; \mathbb{C} \mapsto [\mathcal{O}_C \xrightarrow{\text{id}} \cO], 
	\end{equation*}
which has left and right adjoints $i'^*  \ \dashv i'_*  \ \dashv i'^!$ defined in the same way as in~\eqref{eq-adjoints to i}. Analogously to Lemma~\ref{lem-sod1} this induces the semiorthogonal decomposition
	 \begin{align}\label{sod.3}
	    \mathcal{D}(\mathcal{T}_C) =\ &  \langle j_*\mathcal{D}(C), \ i'_*\cD(\mathcal{V})  \rangle. 
	 \end{align}
Thus any object $T \in \cD(\mathcal{T}_C)$ lies in the distinguished triangle
\begin{equation}\label{rep.3}
    	 i'_*{i'}^!T \overset{ev}{\rightarrow} T \rightarrow j_*j^*T = L_{[\cO \rightarrow \cO]}(T) \xrightarrow{\delta_T} i'_*{i'}^!T[1]. 
\end{equation}

The following Lemma describes each factor in the above two types of semiorthogonal decomposition.

\begin{Lem}\label{lem-simple}
  Any object $T \in \cD(\T)$ is uniquely determined by a triple
$(V,E,\varphi)$, where $V \in \cD(\mathcal{V})$, $E \in \cD(C)$, and
\[
\varphi \colon \cO \otimes V \to E
\]
is a morphism in $\cD(C)$. Conversely, every such triple determines an object of $\cD(\T)$. With this notation, we have
\[
i'^!(T)=V \in \cD(\mathcal{V})
\quad\text{and}\quad
j^*(T)=\cone(\varphi)\in\cD(C).
\]
Moreover, in \eqref{rep.3}, the morphism $\delta_T$ is given by
\[
\delta_T=\delta_{[\cO\otimes V\to 0]}\circ j_*(\varphi'),
\]
where $\varphi'$ denotes the connecting morphism in the distinguished triangle
\[
\cO\otimes V
\xrightarrow{\varphi}
E
\longrightarrow
\cone(\varphi)
\xrightarrow{\varphi'}
\cO\otimes V[1]
\]
in $\cD(C)$.
\end{Lem}
%We sometimes denote the triple \((V, E, \varphi)\) by \([ \cO \otimes V \xrightarrow{\varphi} E ]\). \Ale{why sometimes? i thought always}

\begin{proof}
    %Any object $T \in \cD(\cT_C)$ lies in 
    Consider the unique exact triangle \eqref{rep.1} for $T$ with $ \varphi_T \in \Hom_{ \cD(\cT)}(i_*i^*T, \ j_*j^!T[1])$, then by adjunction 
    \begin{equation*}
       j^* \varphi_T \in %\Hom_{ \cD(\cT)}(i_*i^*T, \ j_*j^!T[1]) =
       \Hom_{\cD(C)}(j^*(i_*i^*T)[-1], j^!T).
    \end{equation*}
    We know 
    $$
    j^*(i_*i^*T)[-1] = j^{!} L_{[\cO \rightarrow \cO]} (i_*i^*T)[-1] \overset{(*)}{ = }j^{!}j_* ( i^*T \otimes \cO ) = i^*T \otimes \cO
    $$
   where $(*)$ follows from injectivity of $[\cO \to 0]$. Therefore, we have
   $$
   j^*\varphi_T \in \Hom_{\cD(C)}(i^*T \otimes \cO,\, j^{!}T)
   $$
   So we set $V:=i^*T$ and $E := j^!T$, and $\varphi : = j^{*} \varphi_T$ is the corresponding morphism.
    
           % so we set $V \otimes \cO := j^*(i_*i^*T)[-1]$ and $E := j^!T$, and $\varphi$ is the corresponding morphism. \color{blue} from easy computation of injective resolution for $j^!T$ we have that $j^*(i_*i^*T)[-1] = i^*T$ indeed. maybe we can include it here as it's indeed less confusing. Moreover, $i^!T$ follows also from it. We can include explicit formula for adjoint of $i_*$ after this lemma then as we fixed the notation. \color{black}
           \medskip 
           
    For the second claim, consider the exact triangle \eqref{rep.1}, and take $i'^! (-) = R\Hom([\cO \to \cO] , - )$. From~\eqref{sod.3} it follows that $R\Hom([\cO \to \cO] , j_* j^! T ) = 0$, thus
     \[
     i'^!T = R\Hom([\cO \to \cO], i_*i^*T)= R\Hom(i^*[\cO \to \cO], i^*T)=R\Hom(\mathbb{C}, V)=V.
     \]
By taking $j^!$ from the exact sequence~\eqref{rep.3}, we obtain
$$
j^!(i'_*{i'}^!T) \xrightarrow{j^!(ev)} j^!T \xrightarrow{} j^!(j_*j^*T).
$$
We claim $j^!(ev) = \varphi$, and hence $j^!(j_*j^*(T)) = \cone(\varphi)$ as required. 

By Lemma~\ref{lem-simple}, the evaluation morphism $\mathrm{ev}\colon i'_*{i'}^!T \to T$
in $\mathcal{D}(\T)$ corresponds to the following commutative diagram in $\mathcal{D}(C)$:

     \[
     \begin{tikzcd}
\cO \otimes  R\Hom(\mathbb{C}, V)\arrow[d, "id"] \arrow[r, "ev"] & \cO \otimes V \arrow[d, "\varphi"] \\
\cO \otimes R\Hom(\mathbb{C}, V) \arrow[r]  & E,       
\end{tikzcd}
     \]
thus the bottom morphism is $\varphi$. By adjunction, it follows that $j^!(ev) = \varphi$, which concludes the proof that $j^*T=\text{cone}(\varphi)$. 

From $j^!(ev) = \varphi$ we also get that $j^!(\delta_T)=\varphi'$. Alongside with $j^!j_*= \text{id}$ we get that both $\delta_T$ and $\delta_{[\cO\otimes V\to 0]}j_*(\varphi')$ go by adjunction to the same morphism, this shows the last part of the claim. 
\end{proof}

%\color{blue} a bit unclear why do we need this sentence. and later it's just a collection of unrelated facts, i don't know how to make it more coherent \color{black}

Note that Lemma \ref{lem-simple} shows that any %arbitrary
object $T \in \cD(T_C)$ can be represented by 
$$
[ \cO \otimes V \xrightarrow{\varphi} E ]
$$
for
a morphism $\varphi$ in $\cD(C)$. 

\bigskip

\subsection{Serre functor}
Since $\cD(\mathcal V)$ and $\cD(C)$ both admit Serre functors, the triangulated category $\cD(\T)$ also admits a Serre functor, which we denote by $S$. The following result was also computed in \cite[Theorem 3.8 and Remark 3.9]{AlexeevKuznetsov2025}.  
\begin{Lem}
    Given an object $T = [\cO \otimes V \xrightarrow{\varphi} E] \in \cD(\mathcal{T}_C)$, we have 
\[
S\left[
\vcenter{
  \hbox{
    $\xymatrix@C=0em@R=2em{
      \cO \otimes V \ar[d]^{\varphi} \\
      E
    }$
  }
}
\right]
=
\left[
\vcenter{
  \hbox{
    $\xymatrix@C=0em@R=2em{
      \cO \otimes \cone\big(R\Hom(\cO, \cone(\varphi) \otimes \omega_C) \xrightarrow{\pi} V\big)  \ar[d]^{\tilde{ev}} \\
      \cone(\varphi) \otimes \omega_C[1]
    }$
  }
}
\right].
\]
Here $\tilde{ev}$ is the induced evaluation map, and $\pi$ is the composition 
$$
R\Hom(\cO, \cone(\varphi) \otimes \omega_C) \xrightarrow{\varphi\otimes \omega_C} (H^0(\omega_C) \otimes V[1] \oplus V) \to  V,
$$
where the second arrow is the projection onto the second factor. As a result, if $\operatorname{cl}(T)=(r,d,n)$, then
\[
  \operatorname{cl}(S(T))=\bigl(n-r,\,-d+2(n-r)(g-1),\,n-d+(n-r)(g-1)\bigr).
\]
\end{Lem}
\begin{proof}
After a mutation of~\eqref{sod.1}, we get the semiorthogonal decomposition
\begin{equation}\label{sod.21}
    \cD(\mathcal{T}_C) = \langle L_{[\cO \rightarrow 0]}j_*\mathcal{D}(C), i_*\cD(\mathcal{V}) \rangle. 
\end{equation}
For simplicity, we denote the functor $j'_* :=  L_{[\cO \rightarrow 0]} \circ j_* \colon \cD(C) \to \cD(\mathcal{T}_C)$.
     
From Lemma~\ref{lem-simple}, we get ${i'}^!T = V$ and $j^*T=\text{cone}(\varphi)$. Thus, we have $S(i'_*{i'}^!T)=i_*V$ and $S(j_*j^*T)=j_*'(\text{cone}(\varphi)\otimes \omega_C[1])$ by~\cite[Section 2.1]{Kuz21:serre-functor}. By applying $S(-)$ to the exact sequence given by SOD~\eqref{sod.21}, we obtain that $S(T)$ fits into the exact sequence 
\[
i_*V \to S(T) \to  j_*'(\text{cone}(\varphi)\otimes \omega_C[1]) \overset{S(\delta_T)}{\to} i_*V[1].
\]
%Since $j_*' = L_{[\cO \rightarrow 0]} \circ j_*$, 
We compute that
\[
j_*'(\text{cone}(\varphi)\otimes \omega_C[1])=\left[R\Hom(\cO, \text{cone}(\varphi)\otimes \omega_C)\overset{ev}{\to} \text{cone}(\varphi)\otimes \omega_C\right][1]. 
\]
It remains to understand the morphism $S(\delta_T)$. We know 
    $$
    \delta_{[\cO \to 0]} \in \Hom([0 \to \cO][1] , [\cO \to \cO][1] ) \cong \C
    $$
and so $S(\delta_{[\cO \to 0]})$ is the unique non-zero map in 
$$
%S(\delta_{[\cO \to 0]}) \in 
\Hom\left(\left[\cO^{h^0(\omega_C)}[2] \oplus \cO[1] \xrightarrow{ev} \omega_C[2]\right], [\cO[1] \to 0]\right). 
$$
On the other hand, from Lemma~\ref{lem-simple} we have 
\begin{equation*}
    S(\delta_T) = S(\delta_{[\cO \otimes V \to 0]}) \circ j'_*S_{\cD(C)}(\varphi').
\end{equation*}
Taking $i^*$ gives 
$$
R\Hom(\cO, \cone(\varphi) \otimes \omega_C[1]) \xrightarrow{\varphi\otimes \omega_C} (H^0(\omega_C) \otimes V[2] \oplus V[1]) \to  V[1],
$$
where the second map is simply projection to the second component. This shows the first part of the claim. 

If $\cl(T) = (r, d, n)$, then
\begin{align*}
    \cl(j_*j^*T) = & \cl(T) - \chi([\cO \rightarrow \cO], T)\cl([\cO \rightarrow \cO])\\
    = & (r-n, d, 0). 
\end{align*}
Similarly, we have 
\begin{align*}
    \cl(j_*'j'^*T) = & \cl(T) - \chi([\cO \rightarrow 0], T)\cl([\cO \rightarrow 0])\\
    = & (r,\ d,\ d+r(1-g) ). 
\end{align*}
%Thus $\cl(j_*(\text{cone}(\varphi)\otimes \omega_C[1])) = \big(n-r,\ -d +2(n-r)(g-1),\ 0 \big)$ and 
Combining those together we obtain 
\begin{equation*}
    \cl(j_*'(\text{cone}(\varphi)\otimes \omega_C[1])) = (n-r,\ -d +2(n-r)(g-1),\ -d +(n-r)(g-1)  ), 
\end{equation*}
which implies the second part of the claim. 

\end{proof}

\subsection{Dual functor} In this section, we define an involutive anti-autoequivalence $\du$ of $\cD(\cT_C)$; that is, $\du^2=\mathrm{id}$.
Consider an embedding of the curve $C$ into a smooth Fano threefold $X$ (e.g. $\PP^3$). Let
$\tilde X$ be the blow-up of $X$ along $C$ and let $\mathbb{E}$ be the exceptional divisor. We have a commutative diagram
\[
\begin{tikzcd}
\mathbb{E} \arrow[d, "p"] \arrow[r, "i"] & \tilde X \arrow[d, "\pi"] \\
C \arrow[r, "j"] & X
\end{tikzcd}
\]
and, by Orlov's blow-up formula, a semiorthogonal decomposition
\[
\cD(\tilde X)=\langle \pi^*\cD(X),\, i_*p^*\cD(C)\rangle.
\]
Since $X$ is Fano, $\mathcal{O}_X$ is exceptional; using $\pi^*\mathcal{O}_X=\mathcal{O}_{\tilde X}$, we refine this to
\[
\cD(\tilde X)=\langle \pi^*(\mathcal{O}_X^{\perp}),\, \mathcal{O}_{\tilde X},\, i_*p^*\cD(C)\rangle.
\]
Set
\[
\ku(\tilde X):={}^{\perp}\!\bigl(\pi^*(\mathcal{O}_X^{\perp})\bigr)=\langle \mathcal{O}_{\tilde X},\, i_*p^*\cD(C)\rangle.
\]
It is shown in~\cite[Lemma~3.4]{AlexeevKuznetsov2025} that $\ku(\tilde X)\simeq \cD(\cT_C)$.
We consider the involutive functor
\[
\du\colon \cD(\tilde X)\to \cD(\tilde X),\qquad
\du(-):=(-)^{\vee}\otimes \mathcal{O}_{\tilde X}(-\mathbb{E}).
\]

\begin{Lem}\label{lem-dual}
    The restriction of $\du$ to $\mathcal{D}(\mathcal{T}_C)$ gives a well-defined functor on $\mathcal{D}(\mathcal{T}_C)$ such that
    \[
    \du([\cO \otimes V \overset{\varphi}{\to} E]) = [\cO \otimes V^{\vee} \xrightarrow{\psi^{\vee}} (\cone(\varphi))^{\vee}[1]], 
    \]
    where $\psi$ fits in the exact triangle in $\cD(C)$
    \begin{equation*}
        \cone(\varphi)[-1] \xrightarrow{\psi} \cO \otimes V \overset{\varphi}{\to} E.
    \end{equation*} 
\end{Lem}
\begin{proof}
   We first compute $\du([\cO \to 0])$. Under the equivalence $\ku(\tilde{X})\simeq \mathcal{D}(\T)$, we have that $[\cO \to 0]$ corresponds to $\mathcal{O}_{\tilde{X}}$. Thus, $\mathbb{D}([\cO \to 0])= \mathcal{O}_{\tilde{X}}(-\mathbb{E}) \in \cD(\tilde{X})$ which lies in the exact triangle 
   $$
   \mathcal{O}_{\tilde{X}}(-\mathbb{E}) \to \mathcal{O}_{\tilde{X}} \to \mathcal{O}_{\mathbb{E}} = i_*p^*\cO.
   $$
 Thus under our correspondence, we get $\du([\cO \to 0]) = [\cO \xrightarrow{\text{id}} \cO] \in \cD(\mathcal{T}_C)$.    

   The next step is to compute $\du([0 \to E])$ for an object $E \in \cD(C)$, that corresponds to 
   \begin{align*}
   \du (i_*p^*E[-1]) & \ = (i_*p^*E[-1])^{\vee} \otimes \mathcal{O}_{\tilde{X}}(-\mathbb{E})\\
   & \ \overset{\text{GV}}{=} i_* \big( (p^*E[-1])^{\vee} \otimes i^*(\mathcal{O}_{\tilde{X}}(-\mathbb{E})) \otimes \omega_{\mathbb{E}} \otimes i^*\omega_{\tilde{X}} [-1]  \big) \\
    & \ = i_* \big( (p^*E[-1])^{\vee} \big)\\
    & \ \overset{}{= } i_*p^* E^{\vee}, 
   \end{align*}
   where by GV we mean Grothendieck--Verdier duality. Under $\ku(\tilde{X})\simeq \mathcal{D}(\T)$ we get that $i_*p^* E^{\vee}$ corresponds to $[0 \rightarrow E^{\vee}[1]] \in \cD(\T)$. Hence, the functor $\du$ preserves $j_*\cD(C)$ and acts by the usual derived dual, shifted by one. Thus, a morphism $f \in \Hom(j_*E, j_*F)$ is sent to
\[
\du(f) = f^{\vee}[1] \colon \du(j_*F) = j_*F^{\vee}[1] \to \du(j_*E) = j_*E^{\vee}[1].
\]
On the other hand, the unique morphism
\[
t \in \Hom\left([\cO[-1] \to 0],\ [0 \to \cO]\right)
\]
is sent to the unique morphism
\[
\du(t) \in \Hom\left([0 \to \cO][1],\ [\cO \to \cO][1]\right).
\]
Applying %Combining
these two observations to the exact sequence~\eqref{rep.1} proves the claim.
 
 %Now consider the object $[\cO \otimes V \overset{\varphi}{\to} E]$
   
\end{proof}

\section{Tilting stability conditions}\label{section tilt}
In this section, we describe a two-dimensional slice of the space of Bridgeland stability conditions on $\cD(\T)$ obtained by tilting the natural heart $\T$ with respect to a torsion pair. The construction is analogous to the surface case first treated by Bridgeland \cite{bridgeland:K3-surfaces}. For definitions and background on (pre-)stability conditions and the support property, see \cite[Appendix~1]{bayer:the-space-of-stability-conditions-on-abelian-threefolds}; for the general theory of tilting see~\cite{happel:tilting-in-abelian-categories-and-quasitilted-algebras}.

\bigskip

We start by extending the classical notion of $\mu$-stability of sheaves on a curve to triples. 
\begin{Def}\label{def-alpha}
Fix $\alpha \in \mathbb{R}_{\geq 0}$. For any object $T = [\cO \otimes V \xrightarrow{\varphi} E] \in \T$, we define the slope 
\begin{equation}\label{mu}
\mu_{\alpha}(T)\ \coloneqq\ \left\{\!\!\!\begin{array}{cc} \frac{\deg(E)}{\rk(E)} + \alpha \frac{\dim V}{\rk(E)}
& \text{if }\rk(E)\ne0, \\
+\infty & \text{if }\rk(E)=0. \end{array}\right.
\end{equation}
We say that $T \in \mathcal{T}_C$ is $\mu_{\alpha}$-(semi)stable if for all non-trivial subobject $0 \neq T' \subset T$ in $\T$, we have $\mu_{\alpha}(T') < (\leq)\ \mu_{\alpha}(T/T')$.
\end{Def}
We call an object 
$T = [\cO \otimes V \xrightarrow{\varphi} E] \in \T$
with \(\operatorname{rk}E>0\) \emph{torsion-free} if \(E\) is a torsion-free sheaf and the induced map
\(H^0(\varphi)\colon V\to H^0(E)\) is injective. 
By definition, any \(\mu_{\alpha}\)-semistable object in \(\cT_C\) of positive rank is torsion-free. From~\cite[Proposition 2.3]{He1996} it follows that every object $T\in\cT_C$ admits a unique Harder–Narasimhan filtration with $\mu_{\alpha}$-semistable factors\footnote{It has been assumed in \cite{He1996} that $\alpha>0$, but the same proof is valid for $\alpha=0$. }. 

For the remainder of this section, we focus on the case $\alpha=0$; we write $\mu:=\mu_{0}$ for simplicity. By truncating the HN filtration of the objects in $\T$ at a real number $b \in \mathbb{R}$ with respect to slope $\mu$, we get a torsion pair. Let $\mathbb{T}^b$ and $\mathbb{F}^b$ be the full subcategories of $\T$ such that $\mathbb{T}^b$ consists of objects whose quotients have slope bigger than $b$, and $\mathbb{F}^b$ consists of objects whose subobjects have slope less than or equal to $b$. Then $(\mathbb{T}^b, \mathbb{F}^b)$ is a torsion pair in $\T$, and so 
\begin{equation*}
        \cA(b) \coloneqq \langle \mathbb{T}^b,\  \mathbb{F}^b[1] \rangle 
    \end{equation*}
is the heart of a bounded $t$-structure. 
Note that any $T\in \mathbb{F}^b$ has $\rd(T)>0$ and is torsion-free.

%\begin{Rem}
%    Note that since $\mu([\cO \to 0])=+\infty$, for any $T= [\cO \otimes V\overset{\varphi}{\to} E]\in \mathcal{T_C}$ which is $\mu$-semistable of $\mu(T)< +\infty$, we have that $H^0(\varphi)$ is injective. In particular it follows that if $T= [\cO \otimes V\overset{\varphi}{\to} E]\in \mathbb{F}^b$, then $H^0(\varphi)$ is injective as well.
%\end{Rem}

To describe our two-dimensional slice of stability conditions, we need to define the Brill-Noether function $\Phi_C \colon \mathbb{R} \to \mathbb{R}$ similar to the Le Potier function on surfaces. We define
\begin{equation*}
    \Phi_C(x) \coloneqq \limsup_{\mu \to x} %\sup_{F \in \Coh(C)} 
    \left\{ \frac{h^0(C, F)}{\rk(F)} \colon\ \text{$F \in \Coh(C)$ is semistable with slope $\mu(F) =\mu$}     \right\}. 
\end{equation*}

\begin{Lem}\label{lem-BN function}
The BN function is well-defined, satisfying $\Phi_C(x) = 0$ if $x<0$, $\Phi_C(x) = x+1-g$ if $x>2g-2$ and $\Phi_C(x) \leq \frac{1}{2}x+1$ if $x \in [0, 2g-2]$. The BN function is the smallest upper semicontinuous function $\Phi$ satisfying 
\begin{equation*}
    \frac{h^0(F)}{\rk(F)} \leq \Phi(\mu(F))
\end{equation*}
for every semistable sheaf $F$ on $C$. 
\end{Lem}
\begin{proof}
    There is a slope-stable rank $r$ and degree $d$ vector bundle on $C$ for any integers $r>0$ and $d$ which are coprime, 
    see~\cite{mumford:slope-stability-projective-invarinats}. Thus for any rational number $\mu$, there is a stable bundle of slope $\mu$. Since Clifford's Theorem gives an upper bound for $\frac{h^0(F)}{\rk(F)}$ for any stable bundle $F$, see~\cite{mercat:clifford-theorem}, the function $\Phi_C$ is well-defined.
\end{proof}
Recall that for any $T=[\cO\otimes V \to F]\in \mathcal{D}(\T)$, we associate a vector $\cl(T)$ in $\mathbb{Z}^3$ as 
\[
\cl(T)= (\rd(T),\dd(T),\nd(T)) = \big(\rk E,\ \deg E,\ \dim_{\C} V\big).
\]
The main goal of this section is to prove the following. 

\begin{Thm}\label{thm-tilting}
    There is a two-dimensional continuous family of stability conditions parametrized by  $(b,w) \in \mathbb{R}^2$ for $w> \Phi_C(b)$ given by $(b,w) \mapsto \sigma_{b,w} \coloneqq (\cA(b), Z_{b,w})$ for the the group homomorphism 
  \begin{equation*}
  Z_{b,w} \colon \mathcal{N}(\cD(T)) \rightarrow \mathbb{C} \;\;\;,\;\;\; Z_{b,w}(T) = -\mathbf n(T)+w\mathbf r(T)+i (\mathbf d(T)-b\mathbf r(T)).
  \end{equation*}
\end{Thm}

In this section, we prove the claim only on the restricted domain \((b,w)\in\mathbb{Q}\times\mathbb{R}_{>0}\). Lemmas~\ref{lem-Z is stability function} and~\ref{lem-pre} prove that $\sigma_{b,w}$ are pre-stability conditions; that is, $Z_{b,w}$ is a stability function on $\mathcal{A}(b)$ satisfying the HN property. And Lemma \ref{lem-support} verifies the
support property. The theorem then follows from the classification in Theorem \ref{thm-main} together with the
deformation theory of Bridgeland stability conditions \cite[Theorem 1.2]{bridgeland:space-of-stability-conditions} or \cite[Theorem~1.2]{bayer:deformation}, which allows us to extend the result to all $b$.

\begin{Lem}\label{lem-Z is stability function}
    The group homomorphism $Z_{b,w}$ is a stability function on $\mathcal{A}(b)$.
\end{Lem}
\begin{proof}
    Recall that being a stability function is equivalent to $Z_{b,w}(T) \in \mathbb{H} \cup \mathbb{R}^{<0}$ for any $0 \neq T \in \cA(b)$, where $\mathbb{H}$ is the upper half-plane $\{\Im z>0\}$. By definition, we know $\Im[Z_{b,w}(T)] \geq 0$. Take $T$ such that $\Im[Z_{b,w}(T)] = 0$, it fits in the exact sequence in $\mathcal{A}(b)$
    \[
    0\to \mathcal{H}^{-1}(T)[1]\to T \to \mathcal{H}^0(T)\to 0.
    \]
    Since $\Im [Z_{b,w}]$ is additive, it follows that $j^!\mathcal{H}^0(T)=0$ and $j^!\mathcal{H}^{-1}(T)$ is a $\mu$-semistable sheaf with $\mu(j^!\mathcal{H}^{-1}(T))=b$. If $\mathcal{H}^{-1}(T)=0$, then $\nd(T)=\nd(\mathcal{H}^{0}(T))>0$ which gives $\Re [Z_{b, w}(T)]<0$.
        
    %First assume $j^!\mathcal{H}^{-1}(T)=0$. \Ale{Then $\mathcal{H}^{-1}(T)=i_*i^*\mathcal{H}^{-1}(T)$, but since $\mathcal{H}^{-1}(T)\in \mathbb{F}^b$, it follows that} $\mathcal{H}^{-1}(T)=0$. Therefore,  \Ale{$\nd(T)=\nd(\mathcal{H}^{0}(T))>0$} which gives $\Re [Z_{b, w}(T)]<0$.
    
    %Now assume $j^!\mathcal{H}^{-1}(T)\neq 0$, then  $\rd(\mathcal{H}^{-1}(T))>0$. 

    Now assume $\mathcal{H}^{-1}(T)\neq 0$. Since $\mathcal{H}^{-1}(T)\in \mathbb{F}^b$, it implies that $\rd(\mathcal{H}^{-1}(T))>0$. As $\Re [Z_{b, w}(\mathcal{H}^0(T))]\leq 0$, it is enough to show that $\Re [Z_{b, w}(\mathcal{H}^{-1}(T)[1])]<0$. Since $\mathcal{H}^{-1}(T)\in \mathbb{F}^b$ and, in particular, torsion-free, we get 
    \[
    \frac{\dim i^* \mathcal{H}^{-1}(T)}{\rd( \mathcal{H}^{-1}(T))} \leq \Phi_C(\mu( \mathcal{H}^{-1}(T)))< w, 
    \]
    which implies $\Re [Z_{b, w}(T_1[1])]<0$ as required. 
\end{proof}

Before proceeding to the proof of the HN property for $Z_{b,w}$, we recall the notion of \(\sigma_{b,w}\)-stability. For any non-zero object $T \in \cA(b)$, we define the slope function 
  \begin{equation*}
      \nu_{b,w}(T) = -\frac{\Re[Z_{b,w}(T)]}{\Im[Z_{b,w}(T)]} = \frac{\nd(T)-w \rd(T)}{\dd(T)-b  \rd(T)}. 
  \end{equation*}
  From Lemma~\ref{lem-Z is stability function}, we have $\dd(T)-b  \rd(T) \geq 0$, and if it is zero, then we set $\nu_{b,w}(T) = +\infty$. 
  \begin{Def}
We say $0\neq T \in \cD(\cT_C)$ is $\sigma_{b,w}$-(semi)stable if and only if
\begin{itemize}
\item $T[k]\in\cA(b)$ for some $k\in\Z$, and
\item $\nu\_{b,w}(T')\,< (\le)\,\nu\_{b,w}\big(T[k]/T')$ for all non-trivial subobjects $T'\hookrightarrow T[k]$ in $\cA(b)$.
\end{itemize}
\end{Def} 
\begin{Lem}\label{lem-pre}
The stability function $Z_{b,w}$ on $\mathcal{A}(b)$ has the Harder--Narasimhan property when $b \in \mathbb{Q}$.
    %The pair $\sigma_{b, w} = (\cA(b), Z_{b,w})$ is a pre-stability condition when $b \in \mathbb{Q}$. 
\end{Lem}
\begin{proof}
    It is enough to verify that $\mathcal{A}(b)$ satisfies the chain conditions of~\cite[Proposition~2.4]{bridgeland:space-of-stability-conditions}. Since $\Im [Z_{b,w}]$ is discrete when $b$ is rational and $\T$ is noetherian, following the proof of~\cite[Proposition~7.1]{bridgeland:K3-surfaces}, it suffices to show that for any $T \in \mathcal{A}(b)$ there is no infinite filtration in $\mathcal{A}(b)$
    \[
    0=A_0 \subsetneq A_1 \subsetneq \dots \subsetneq A_k\subsetneq \dots \subsetneq T, 
    \]
    such that $\Im [Z_{b, w}(A_k)]=0$ for all $k$. From the proof of Lemma~\ref{lem-Z is stability function} it follows that  $j^!\mathcal{H}^0(A_k)=0$ for any $k$. Denote $Q_k = T/A_k$. Following~\cite[Lemma 6.17]{macri:intro-bridgeland-stability} we may assume $\mathcal{H}^0(Q_{k-1})=\mathcal{H}^0(Q_{k})$ and $\mathcal{H}^{-1}(A_{k-1})=\mathcal{H}^{-1}(A_{k})$ for all $k$. So there is the following long exact sequence of cohomology for any $k$
    \begin{equation}\label{eq les hn filtration}
        0\to \mathcal{H}^{-1}(A_k)\to \mathcal{H}^{-1}(T) \to \mathcal{H}^{-1}(Q_k) \to \mathcal{H}^{0}(A_k)\to \mathcal{H}^{0}(T)\to \mathcal{H}^{0}(Q_k)\to 0.
    \end{equation}
    By taking $j^!$ of it we get that $j^!\mathcal{H}^{-1}(Q_{k-1}) = j^!\mathcal{H}^{-1}(Q_k)$ for all $k$. Since $\mathcal{H}^{-1}(Q_k)$ is torsion-free, we have that $\dim i^*\mathcal{H}^{-1}(Q_k)$ is bounded by $\dim i^*\mathcal{H}^{-1}(Q_k) \leq h^0(C, j^!\mathcal{H}^{-1}(Q_k))$. Therefore, $\dim i^* \mathcal{H}^{0}(A_k)$ has only finitely many possibilities for all $k$. Combining with $j^!\mathcal{H}^{0}(A_k) =0$ and $\mathcal{H}^{-1}(A_{k-1})=\mathcal{H}^{-1}(A_{k})$, we obtain that there is no infinite sequence like above, this shows the claim.
\end{proof}

To prove the support property, we first analyze the large-volume limit along vertical lines.

\begin{Lem}\label{lem-vertical-large-volume-limit}
   Fix $b \in \mathbb{Q}$. If $T\in \mathcal{A}(b)$ is $\sigma_{b, w}$-semistable for all $w \gg 0$, then it satisfies one of the following conditions 
    \begin{enumerate}
        \item $\mathcal{H}^{-1}(T)=0$ and $\mathcal{H}^0(T)$ is $\mu$-semistable,
        \item $j^!\mathcal{H}^0(T)=0$ and $\mathcal{H}^{-1}(T)$ is $\mu$-semistable. 
    \end{enumerate}
\end{Lem}
\begin{proof}
First assume $\cH^{-1}(T)=0$ and that $\cH^0(T)$ is $\mu$-unstable, so it fits into a short exact sequence
\begin{equation}\label{eq}
0 \to T'' \to T=\cH^0(T) \to T' \to 0,
\end{equation}
in $\cT_C$, where $T'$ is $\mu$-semistable and $\mu(T'')>\mu(T')=\mu^{-}(\cH^0(T))>b$. This implies that \eqref{eq} is also a short exact sequence in the heart $\cA(b)$. Moreover, for objects of positive rank, the ordering by $\nu_{b,w}$-slope agrees with the ordering by $\mu$-slope, because
\[
\lim_{w\to\infty}\frac{\nu_{b,w}(T)}{w}
=
\left(b-\frac{\dd(T)}{\rd(T)}\right)^{-1}.
\]
Hence $\sigma_{b,w}$-semistability of $T$ for $w \gg0 $ gives a contradiction. Therefore, $\cH^0(T)$ is $\mu$-semistable, as claimed in part~(a).

Now suppose \(\mathcal{H}^{-1}(T)\neq 0\). Recall that it implies that $\rd(\mathcal{H}^{-1}(T))>0$. We claim that \(j^!\mathcal{H}^0(T)=0\), or equivalently
\(\Im [Z_{b,w}\big(\mathcal{H}^0(T)\big)]=0\).
Otherwise, taking cohomology yields a short exact sequence in \(\mathcal{A}(b)\)
\[
0 \longrightarrow \mathcal{H}^{-1}(T)[1] \longrightarrow T \longrightarrow \mathcal{H}^0(T) \longrightarrow 0 .
\]
Then
\[
\lim_{w\to\infty}\Re [Z_{b,w}\!\big(\mathcal{H}^{-1}(T)[1]\big)]
=-\infty < -\nd(\mathcal{H}^0(T)) \leq \lim_{w\to\infty}\Re [Z_{b,w}\!\big(\mathcal{H}^{0}(T)\big)] ,
\]
which implies \(\nu_{b,w\gg 0}\!\big(\mathcal{H}^{-1}(T)[1]\big)>\nu_{b,w\gg 0}\!\big(\mathcal{H}^0(T)\big)\), a contradiction to the \(\sigma_{b,w}\)-semistability of \(T\). 

Finally, for any subobject \(T'\hookrightarrow \mathcal{H}^{-1}(T)\) in $\T$ %\(\mathcal{A}(b)\) 
we have
\(\mu^{+}(T')\le \mu^{+}\big(\mathcal{H}^{-1}(T)\big)\).
Hence \(T'[1]\) is a subobject of \(T\) in \(\mathcal{A}(b)\), and the \(\mu\)-semistability of \(\mathcal{H}^{-1}(T)\) follows by the same argument as in part~(a).
    %Recall that $T$ fits in exact triangle $T_1[1]\to T \to T_2$, where $T_1\in \mathbb{F}^b$ and $T_2\in \mathbb{T}^b$. We write $T_i = [\cO \otimes V_i \to E_i]$. 
%If $\Im[Z_{b, w}(T)]=0$, then $E_2=0$ and $E_1$ is $\mu$-semistable of slope $b$ or 0, thus the claim follows. So we assume for the rest that $\Im [Z_{b, w}(T)]>0$.
%Assume that $r\geq 0$. $\Im [Z_{b, w}(T)]>0$ implies that $(b-\frac{d}{r})^{-1}<0$. By definition, we also have $\mu(T_1[1])\leq b$. Since $T$ is $\sigma_{b, w}$-semistable for $w \gg 0$ if follows $T_1=0$, so $T\in \mathcal{T}_C$. If $T$ is not $\mu$-semistable, then there exists a subobject $T'\subset T$ in $\mathcal{T}_C$ such that $\mu(T')>\mu(T)>b$. Thus, $(b-\mu(T'))^{-1}>(b-\mu(T))^{-1}$ which contradicts with the fact that $T$ is $\sigma_{b, w}$-semistable for $w \gg 0$. 
%Assume that $r<0$, so $(b-\frac{d}{r})^{-1}>0$. If $\Im [Z_{b, w}(T_2)]=0$, then $E_2=0$, thus if $T_1$ is $\mu$-semistable we get the case (b). If $\Im [Z_{b, w}(T_2)]>0$, then $\left(b-\frac{\dd(T_2)}{\rd(T_2)}\right)^{-1}\leq 0$ which contradicts with the assumption that $T$ is $\sigma_{b, w}$-semistable for $w \gg 0$ unless $T_2=0$. Thus, it remains to show that $T_1$ is $\mu$-semistable in both cases. If it is not, there exists a subobject $T_1'\subset T_1$ with $\mu(T_1')>\mu (T_1)$ in $\mathbb{F}^b$. Composing it with the injective map $T_1[1]\hookrightarrow T$, we get an injective map $T_1'[1]\hookrightarrow T$ in $\mathcal{A}(b)$ such that $(b-\mu(T_1'))^{-1}> (b-\frac{d}{r})^{-1}$ which contradicts with the fact that $T$ is $\sigma_{b, w}$-semistable for $w \gg 0$.
\end{proof}

\begin{Lem}\label{lem-support}
       The pre-stability condition $\sigma_{b_0, w_0} = (\cA(b_0), Z_{b_0,w_0})$ satisfies the support property when $b_0 \in \mathbb{Q}$.
\end{Lem}
\begin{proof}
    By \cite[Lemma 11.4]{bayer:the-space-of-stability-conditions-on-abelian-threefolds}, we only need to find a quadratic form $Q$ on $\mathbb{Z}^3$ so that (i) kernel of $Z_{b_0,w_0}$ is negative definite with respect to $Q$, and (ii) any $\sigma_{b_0, w_0}$-semistable object $T \in \cA(b_0)$ satisfies $Q(\cl(T)) \geq 0$. As noted in \cite[Remark 3.5]{li-albanese}, there is always $\delta>0$ satisfying 
    \begin{equation*}\label{eq.3}
\delta^{-1}(x-b_0)^2+w_0-\delta > \Phi_C(x).        
    \end{equation*}
Then we can consider the quadratic form 
\begin{equation}\label{eq quadratic form}
   Q( r, d, n) = \delta^{-1}( d-b_0 r)^2 +  r^2(w_0-\delta) - n r ,
\end{equation}
which clearly satisfies condition (i). To prove (ii) we apply induction over $\Im[Z_{b_0, w_0}(T)]$. Note that if $\rd(T)=0$, then clearly $Q(\cl(T))\geq 0$, thus we assume $r\neq 0$ and rewrite \eqref{eq quadratic form} as
\begin{equation}\label{eq.2}
\frac{Q( r, d, n)}{ r^2}=\delta^{-1}\left(\frac{ d}{ r}-b_0\right)^2 +(w_0-\delta)-\frac{n}{r}> \Phi_C\left(\frac{ d}{ r}\right)-\frac{ n}{ r}.    
\end{equation}
If $\Im[Z_{b_0, w_0}(T)]$ is zero or minimal, then $T$ is $\sigma_{b_0, w \gg 0}$-semistable. Thus, from Lemma~\ref{lem-vertical-large-volume-limit} and \eqref{eq.2} we get $Q(\text{cl}(T))\geq 0$. Now let $T \in \cA(b_0)$ be an arbitrary $\sigma_{b_0, w_0}$-semistable object which is not $\sigma_{b_0, w \gg 0}$-semistable. Note that as $w$ increases, all quotient and subobjects of $T$ have $\Im[Z_{b_0, w}]$ strictly less then $T$. So, by the inductive assumption, they satisfy the support property. Following~\cite[Proposition 9.3]{bridgeland:K3-surfaces}, we get that $T$ satisfies well-behaved wall-crossing. Thus, there is a wall on which $T$ is strictly $\sigma_{b_0, w}$-semistable, let $T_1\to T \to T_2$ be a destabilizing sequence. From the inductive assumption, we get $Q(\cl(T_i))\geq 0$. Thus, from~\cite[Lemma 3.7]{bayer:the-space-of-stability-conditions-on-abelian-threefolds}, it follows that $Q(\cl(T))\geq 0$ as well. 
\end{proof}

\section{Gluing stability conditions}\label{sec gluing}
In this section, we first review the gluing of stability conditions along a semiorthogonal decomposition, as investigated in~\cite{Collins-gluing-stability-conditions}, and then apply it to our category \(\cD(\cT_C)\). From now on, we assume that the genus of $C$ satisfies $g(C)>0$.
\medskip

Consider a semiorthogonal decomposition of a triangulated category $\mathcal{D}  = \langle \mathcal{D}_1, \mathcal{D}_2 \rangle$. Let $i_1^*$ be the right adjoint functor to the inclusion $i_1 \colon \mathcal{D}_1 \to \mathcal{D}$ and $i_2^!$ be the left adjoint functor to the inclusion $i_2 \colon \mathcal{D}_2 \to \mathcal{D}$. Let $\sigma_i = (\mathcal{A}_i, Z_i)$ be stability conditions on $\mathcal{D}_i$ for $i=1,2$ satisfying $\Hom^{\leq 0}(i_1\mathcal{A}_1, i_2\mathcal{A}_2)=0$. We define
  \[
  \gl (\mathcal{A}_1, \mathcal{A}_2):=\{X \in \mathcal{D} \colon i^*_1 X \in \mathcal{A}_1, i_2^! X \in \mathcal{A}_2 \}.
  \]
  It is shown in~\cite[Lemma 2.1]{Collins-gluing-stability-conditions} that $\gl (\mathcal{A}_1, \mathcal{A}_2)$ is a heart of a bounded $t$-structure on $\mathcal{D}$.

  We say that a stability condition $\sigma = (\mathcal{A}, Z)$ on $\cD$ is \emph{glued} from $\sigma_1$ and $\sigma_2$, and write $\sigma = \gl(\sigma_1, \sigma_2)$, if the heart $\cA$ is given by $\gl(\mathcal{A}_1, \mathcal{A}_2)$ and the stability function is
\[
  Z = Z_{\gl}(E) := Z_1(i_1^*E) + Z_2(i_2^!E)
  \quad \text{for all } E \in \cD.
\]
The following proposition characterizes glued stability conditions.

\begin{Prop}\label{prop 2.2 cp10}
\cite[Proposition~2.2(1)]{Collins-gluing-stability-conditions}
Let $\sigma = (\mathcal{A}, Z)$ be a stability condition on $\mathcal{D}$, and let $\sigma_i = (\cA_i, Z_i)$ be stability conditions on $\mathcal{D}_i$ for $i = 1,2$ such that 
$\cA_i \subset \cA$ for $i=1,2$, $\Hom^{\le 0}(\cA_1, \cA_2) = 0$, and $Z_i = Z|_{\mathcal{D}_i}$. 
Then $\sigma = \gl(\sigma_1, \sigma_2)$.  
\end{Prop}

The converse also holds under stronger Hom-vanishing conditions.

\begin{Prop}\label{thm 3.6 cp10}
\cite[Theorem~3.6]{Collins-gluing-stability-conditions}
Let $(\sigma_1, \sigma_2)$ be a pair of stability conditions on $\mathcal{D}_1$ and $\cD_2$ with slicing $\mathcal{P}_i$ for $i=1, 2$. Let $a$ be a real number in $(0,1)$ such that 
\begin{enumerate}
    \item $\Hom^{\le 0}\big(\mathcal{P}_1(0,1], \mathcal{P}_2(0,1]\big) = 0$, and 
    \item $\Hom^{\le 0}\big(\mathcal{P}_1(a,a+1], \mathcal{P}_2(a,a+1]\big) = 0$.
\end{enumerate}
Then there exists a glued pre-stability condition $\sigma = \gl(\sigma_1, \sigma_2)$ on $\cD$. 
\end{Prop}

%\begin{proof}
%By \cite[Proposition 3.5]{Collins-gluing-stability-conditions}, $\sigma$ is a pre-stability condition. Moereover, the aummed hom-vansihing implies that any $\sigma$-stable object in $\gl(\cA_1, \cA_2)$ either lies in $\cA_1$ or $\cA_2$, so support propetry on $\sigma_1$ and $\sigma_2$ directly implies support propetry for $\sigma$.     
%\end{proof}

\subsection*{First type of gluing}

For our category $\cD(\cT_C)$, we first consider the semiorthogonal decomposition
\begin{equation}\label{s.1}
    \cD(\cT_C) = \langle i_*\cD(\cV),\, j_*\cD(C) \rangle.
\end{equation}
Recall that $\sigma_{\cV}$ denotes the trivial stability condition on $\cD(\cV)$, whose heart and central charge are given by
\[
  \cA_{\cV} = \{\, \C^{\oplus n} \,\}_{n \ge 0}, 
  \qquad 
  Z_{\cV}(n) = -n.
\]
On $\cD(C)$, we consider the stability condition 
\[
  \sigma_{\mu} = (\Coh(C), Z_{\mu}), \qquad 
  Z_{\mu} = -\deg + i\rk,
\]
with corresponding slicing is denoted by $\mathcal{P}_{\mu}$.  
We then define
\[
  \Coh^{x}(C) := \mathcal{P}_{\mu}(x, x+1]
  \qquad \text{for } x \in \mathbb{R}.
\]
Indeed, for $x = \theta + n$ with $n \in \Z$ and $\theta \in [0,1)$, we have
\[
  \Coh^{x}(C) = \Coh^{\theta}(C)[n].
\]
For any $g = (T,f) \in \GL \simeq \Stab(\cD(C))$, we set $ \sigma_g := \sigma_{\mu} \cdot g$, which corresponds to the stability condition 
\begin{equation}\label{eq_sigma_g}
      \sigma_g = \big(\Coh^{f(0)}(C),\, T^{-1} \circ Z_{\mu}\big).
\end{equation}

\begin{Prop}\label{prop-gluing-type 1}
Take $g = (T, f) \in \GL$.
Then there exists a stability condition glued from $\sigma_{\mathcal{V}}$ and $\sigma_g$ with respect to the semiorthogonal decomposition~\eqref{s.1}, 
denoted by $\gl^{(1)}(\sigma_{\mathcal{V}}, \sigma_g)$ if and only if $f(0) < \tfrac{1}{2}$. 
\end{Prop}
\begin{proof}
The condition $f(0)<\tfrac{1}{2}$ is necessary in order to obtain the vanishing
\[
\Hom^{\leq 0}([\cO \to 0], \Coh^{f(0)}(C))=0,
\]
which is required to define the heart $\gl(\mathcal{A}_{\mathcal{V}}, \Coh^{f(0)}(C))$. 
    It also guarantees that the objects in $\Coh^{f(0)}(C)$ are of the form $F[p]$, 
where either $p \le 0$, or $p = 1$ and $\mu^{+}(F) < 0$. 
Hence, the assumptions of Proposition~\ref{thm 3.6 cp10} are satisfied, 
and thus $\gl^{(1)}(\sigma_{\mathcal{V}}, \sigma_g)$ defines a pre-stability condition. It remains to prove the support property, which we divide into three cases. 

If $f(0) < -\frac{1}{2}$, then all indecomposable objects in  $\gl(\mathcal{A}_{\mathcal{V}},\Coh^{f(0)}(C))$ lie either in $i_* \mathcal{A}_{\mathcal{V}}$ or in $j_*\Coh^{f(0)}(C)$ as there are no non-trivial extensions between them, and so the support property follows automatically. 

Now suppose $f(0) \in [-\tfrac{1}{2}, 0)$. Then there exists $b \in \mathbb{R}_{\ge 0}$ such that 
\[
  \Coh^{f(0)}(C) = \langle \mathbb{F}^b,\, \mathbb{T}^b[-1] \rangle,
\]
where $\mathbb{F}^b$ consists of sheaves $F$ on $C$ with $\mu^{+}(F) \le b$, and $\mathbb{T}^b$ consists of sheaves $F$ on $C$ with $\mu^{-}(F) > b$.  
Moreover, up to a $\GL$-action, we may assume
\[
  Z_{\gl}(T) = -\nd(T) - \rd(T)w + i\big(-\dd(T) + b\,\rd(T)\big)
\]
for some $w \in \mathbb{R}$.  
Note that, since the stability function
\[
  -\rd(T)w + i\big(-\dd(T) + b\,\rd(T)\big)
\]
 on $\mathcal{D}(C)$ is obtained from $Z_{\mu}$ by a $\GL$-action, we have $w > 0$.  

%If $b = 0$, 
We consider the quadratic form $Q(r, d, n) = n r$. % and if $b > 0$, we consider $Q (r, d, n) = n d$.  
Clearly, $\ker Z_{\gl}$ is negative definite with respect to $Q$. Moreover, any object $T$ in the glued heart lies in the exact sequence
$$\mathcal{H}^0(T)\to T \to [0\to F][-1]$$ 
as $i^*T\in \mathcal{V}$. Since $\Hom([0\to F][-1], \mathcal{H}^0(T)[1])=\Hom(F, j^{!}\mathcal{H}^0(T)[2])=0$, any stable object $T$ is either of them form  $T = [0\to F][-1]$ or $T = \mathcal{H}^0(T)$, and for both we have $Q(T)=\nd (T) \rd(T) \geq 0$.

Lastly, we consider $f(0)\in [0, \frac{1}{2})$. Similarly, if $f(0) = 0$, then, up to a $\GL$-action, we may assume
\[
  Z_{\gl}(T) = -\nd(T) - \dd(T)\,\alpha + i\,\rd(T)
\]
for some $\alpha \in \mathbb{R}_{>0}$, and if $f(0) \in (0, \tfrac{1}{2})$, then
\[
  Z_{\gl}(T) = -\nd(T) + \rd(T)\,w + i\big(\dd(T) - b\,\rd(T)\big)
\]
for some $b \in \mathbb{R}_{<0}$ and $w \in \mathbb{R}_{>0}$.  
Then by applying a similar argument as in Lemma~\ref{lem-support}, one can show that the support property holds with respect to the quadratic form $Q(r,d,n) = nd$. Namely, we focus on rational values of $b$ and prove the claim by induction on the imaginary part: when $w \gg 0$, we recover $\mu$-stability of objects in $\cT_C$. The final claim then follows from the deformation of stability conditions as discussed in \cite[Theorem 1.2]{bridgeland:space-of-stability-conditions} and the classification of stability conditions in Theorem \ref{thm-main}.

    \end{proof}
%\color{blue} I don't understand why $Q=nd$ and not $Q=nr$ in the second case ($f(0) \in [-\tfrac{1}{2}, 0)$, $b>0$), ask. I still don't understand. When $b>0$, we have that $\ker Z$ consists of classes such that $d=br, n=-rw$, thus $r\neq 0$. Then we get $Q=nr=-wr^2<0$. Any object in the glued heart is of the form $\mathcal{H}^0(T)\to T \to [0\to F][-1]$ (as $i^*T\in \mathcal{V}$). Since $\Hom([0\to F][-1], \mathcal{H}^0(T)[1])=\Hom(F, j^{!}\mathcal{H}^0(T)[2])=0$, we have that the stable are either $[0\to F][-1]$ or $\mathcal{H}^0(T)$. For both $Q=nr\geq 0$. \color{black}

The next corollary shows that half of the tilting stability conditions from Section~\ref{section tilt} are, in fact, also of gluing type.
\begin{Cor}\label{cor-tilting is gluing}
Let $\sigma_{b,w}$ be a tilting stability condition as in Section~\ref{section tilt} with $b<0$. 
Then $\sigma_{b,w}$ coincides with $\gl^{(1)}(\sigma_{\cV},\sigma_{g})$, where $g=(T,f)\in\GL$ is given by 
\[
T^{-1}=\begin{pmatrix}
0 & w \\
-1 & -b
\end{pmatrix},
\qquad
f(0)=-\frac{1}{\pi}\text{arctan}(b).
\]
Moreover, the converse holds, namely if $\sigma_{b,w}$ belongs to $\GL$-orbit of $\gl^{(1)}(\sigma_{\cV},\sigma_{g})$ for some $g\in \GL$, then $b<0$.
\end{Cor}
\begin{proof}
    Take $g=(T, f)\in \GL$ as in the statement. Since $b< 0$, we have 
    \[
    \Hom^{\leq 0}([\cO \to 0], j_*\Coh^{f(0)}(C))=\Hom^{\leq -1}(\cO , \Coh^{f(0)}(C))=0, 
    \]
    as if $F[1]\in \Coh^{f(0)}(C)$, then $\mu^{+}(F)\leq b$. Thus, the first part of the claim follows from Proposition~\ref{prop 2.2 cp10}.

For the second part, let $\sigma_{b,w}$ belong to the $\GL$-orbit of $\gl^{(1)}(\sigma_{\cV},\sigma_{g})$ for some $g\in \GL$. Then $j_*\cO$ is $\gl^{(1)}(\sigma_{\cV},\sigma_{g})$-stable by~\cite[Proposition 2.2(3)]{Collins-gluing-stability-conditions}. But if $b\geq 0$, then the short exact sequence
\[
[\cO \to 0]\to [0\to \cO][1]\to [\cO \to \cO][1]
\]
in $\cA(b)$ makes $j_*\cO[1]$ either strictly $\sigma_{b,w}$-semistable (for $b=0$) or unstable (for $b>0$), and so the second claim follows.
    %with $\phi_{\gl}(j_*\cO [1])>\phi_{\gl}([\cO \to 0])$ from~\cite[Proposition 2.2(3)]{Collins-gluing-stability-conditions}. Which implies that $j_*\cO$ is $\sigma_{b,w}$-semistable as well and the same inequality for phases holds. Consider the following exact triangle in $\mathcal{D}(\T)$ 
    %\[
    %[\cO \to 0]\to [0\to \cO][1]\to [\cO \to \cO][1]. 
    %\]
    %When $b\geq 0$, it gives a short exact sequence in $\mathcal{A}(b)$. From the $\sigma_{b,w}$-semistability of $j_*\cO$, it follows that $\phi_{\sigma}([\cO \to 0])\leq \phi_{\sigma}(j_*\cO[1])$. It gives a contradiction with the inequality of phases above, thus $b<0$.
\end{proof}
\subsection*{Second type of gluing}
Now we consider the second semiorthogonal decomposition
\begin{equation}\label{s.2}
    \cD(\cT_C) = \langle j_* \cD(C),\, i_*' \cD(\cV) \rangle.
\end{equation}
Applying a similar argument as in Proposition~\ref{prop-gluing-type 1} implies the following.

\begin{Prop}\label{prop-gluing-type 2}
Take $g = (T, f) \in \GL$.
Then there exists a stability condition glued from $\sigma_g$ and $\sigma_{\cV}$ with respect to the semiorthogonal decomposition~\eqref{s.2}, 
denoted by $\gl^{(2)}(\sigma_g, \sigma_{\cV})$ if and only if  $f(0) \geq  \tfrac{1}{2}$. 
\end{Prop}
\begin{proof}
    The condition $f(0)\geq \tfrac{1}{2}$ implies that any object in $\Coh^{f(0)}(C)$ is of the form $F[p]$, where $p\geq 1$, or $p=0$ and $\mu^-(F)>0$. This guarantees that the assumptions of Proposition~\ref{thm 3.6 cp10} are satisfied, so $\gl^{(2)}(\sigma_g, \sigma_{\mathcal{V}})$ defines a pre-stability condition. The inequality $f(0)\geq \tfrac{1}{2}$ is also necessary, to obtain the vanishing 
    \[
    \Hom^{\leq 0} (\Coh^{f(0)}(C), [\cO \to \cO])=0. 
    \]

    Since $\phi_{\gl}([\cO \to \cO])=1$, if $T\in \gl(\Coh^{f(0)}(C), \mathcal{A}_{\mathcal{V}})$ is stable and not equal to $[\cO \to \cO]$, then
    \[
    0= \Hom([\cO \to \cO], T) = \Hom(\mathbb{C}, i'^!T)=\mathcal{H}^0(i'^!T), 
    \]
    so $T = j_*j^*T$. Thus, all $\gl^{(2)}(\sigma_g, \sigma_{\mathcal{V}})$-stable objects lie either in $i'_*\mathcal{A}_{\mathcal{V}}$ or $j_*\Coh^{f(0)}(C)$ and so the support property follows automatically.
\end{proof}

\section{An open locus of the stability manifold}\label{sec main thm}
In this section we investigate the open subset
\(\mathrm{Stab}^{\circ}\big(\mathcal{D}(\mathcal{T}_C)\big)\subset \mathrm{Stab}\big(\mathcal{D}(\mathcal{T}_C)\big)\),
described in the Introduction, and prove the classification theorem
(Theorem~\ref{thm-main}), which restates Theorem~\ref{thm-main-intro} from the Introduction. Recall that $\mathrm{Stab}^{\circ}\big(\mathcal{D}(\mathcal{T}_C)\big)$ consists of stability conditions $\sigma$ such that $[\cO \to 0]$ and $[0\to \cK]$ are $\sigma$-stable for all $x\in C$.

\begin{Thm}\label{thm-main}
Up to the action of $\GL$, any stability condition $\sigma \in \mathrm{Stab}^{\circ}\big(\mathcal{D}(\mathcal{T}_C)\big)$ is of one of the following types:  
  \begin{enumerate}
      \item[\textbf{Type A.}] $\sigma$ is the gluing $\gl^{(1)}(\sigma_{\mathcal{V}} ,\sigma_g)$ where $g= (T, f) \in \GL$ with $f(0) < \frac{1}{2}$ and $\sigma_{\mathcal{V}}$ is the stability condition on $\cD(\mathcal{V})$ with the heart $\cA_{\cV}= \{\mathbb{C}^{\oplus n}\}_{n\geq 0}$ and stability function $Z_{\cV}(n) = -n$.  
   
   \item[\textbf{Type B.}] $\sigma$ is the tilting stability condition $\sigma_{b,w} =(\cA(b), Z_{b,w})$ for some $b, w \in \mathbb{R}$ such that $w > \Phi_C(b)$. 
   \end{enumerate}
  \end{Thm}

Pick a stability condition $\sigma = (\cA, Z) \in \stab^{\circ}(\cD(\mathcal{T}_C))$. Up to the $\GL$-action, we may assume $[\cO \to 0]$ is $\sigma$-stable of phase one. By a similar argument as in~\cite[Proposition 2.9]{li-albanese}, we can assume that $j_*\cK$ are all $\sigma$-stable of the same phase.
\begin{Lem}\label{lem-vanishing of cohomology}
There exists $n \geq 0$ such that $j_*\cK[-n] \in \cA$. Moreover, if $T \in \cA$ is a $\sigma$-stable object not isomorphic to $[\cO \to 0]$ or to $j_*\cK[-n]$ for any $x\in C$, then it satisfies the following:
\begin{equation*}
     \cH^{\geq n+1}(j^*T) =  \cH^{\leq n-2}(j^!T) = \cH^{\geq 1}(i^*T) =0.
\end{equation*}
In particular, we get $\cH^{\geq n+ 1}(T) = 0$, $\cH^{\geq n+ 1}(j^!T) = 0$ and $\cH^{\leq -2}(i^*T) = 0$. 
\end{Lem}
\begin{proof}
    We know $\Hom([\cO \to 0], j_*\cK[1]) \neq 0$, so $0 < \phi_{\sigma}(j_*\cK)$ and thus $j_*\cK[-n] \in \cA$ for some $n \geq 0$. Now take a $\sigma$-stable object $T$ as in the statement, then for every $p>0$ and any $x\in C$, we have 
    \begin{equation*}
         \Hom(j^*T,\cK [-n-p] ) = \Hom(T[p] , j_*\cK [-n]) = 0,
    \end{equation*}
    \begin{equation*}
        \Hom(j^!T, \cK[1-n+p]) =\Hom(\cK , j^!T[n-p])  = \Hom(j_*\cK[-n], T[-p]) = 0,
    \end{equation*}
    \begin{equation*}
         \Hom(i^*T, \mathbb{C}[-p]) = \Hom(T ,[\cO \to 0][-p]) = 0,
    \end{equation*}
    These imply that $\mathcal{H}^{\geq n+1}(T)=0$, and so Lemma~\ref{lem-cohomology} yields $\mathcal{H}^{\geq n+1}(j^!T)=0$.

It remains only to show $\cH^{\leq -2}(i^*T) = 0$. Suppose not, take the highest $p \geq 2$ such that $\cH^{-p}(i^*T) \neq 0$. Then there is a non-zero map $[\cO \to 0][p] \to i_*i^*T$. Since $\Hom([\cO \to 0][p], j_*j^!T[1]) =0$ as $\cH^{\leq -2}(j^!T) =0$, taking $\Hom([\cO \to 0][p], -)$ from $T \to i_*i^*T \to j_*j^!T[1]$ implies that $\Hom([\cO \to 0][p], T) \neq 0$ which is not possible as $T\in \mathcal{A}$.
\end{proof}
We investigate three cases separately depending on the phase of $j_*\cK$.
%We first investigate the case of $n=0$, and then discuss $n\geq 1$.  

\subsection*{Case (I)} Suppose $n=0$ and $\phi_{\sigma}(j_*\cK) <1$. 
\begin{Lem}\label{lem-case 0-heart}Take a $\sigma$-stable object $T \in \cA$ which is not isomorphic to $[\cO \to 0]$ or to $j_*\cK$ for any $x \in C$. Then we have $\cH^{p}(T) = \cH^p(j^!T) = \cH^p(i^*T) = 0$ if $p \neq 0, -1$. If $\cH^{-1}(j^!T) \neq 0$, then it is a torsion-free sheaf. Moreover, if $T$ has phase one, then $T = [\cO \otimes {V} \xrightarrow{\varphi} E][1]$ so that $H^0(\varphi)$ is injective.
\end{Lem}
\begin{proof} From Lemmas~\ref{lem-vanishing of cohomology} and~\ref{lem-cohomology}, we immediately get the vanishing $\cH^{p}(T) = \cH^p(j^!T) = \cH^p(i^*T) = 0$ if $p \neq 0, -1$.
%We first show $\cH^1(j^!T) = 0$. If not, then the composition 
%$$
%j_*j^!T \to j_*\cH^1(j^!T)[-1] \to j_*\cK[-1]
%$$
%is not zero. But since $\Hom(i_*i^*T[-1], j_*\cK[-1]) = 0$, taking $\Hom(-, j_*\cK[-1])$ from the exact triangle $i_*i^*T[-1] \to j_*j^!T \to T$ implies that $\Hom(T, j_*\cK[-1]) \neq 0$, a contradiction. 
From the inequality of phases it follows $\Hom(j_*\cK,T[-1])=\Hom(\cK, j^!T[-1]) =0$ for any $x\in C$ which implies that $\cH^{-1}(j^!T)$ is torsion-free.

 If $T$ is of phase one but not equal to $[\cO \to 0]$ or $j_*\cK$, we have 
    \begin{equation*}
        \Hom(T, [\cO \to 0] ) = \Hom(i^*T , \mathbb{C} ) =0, 
    \end{equation*}
    which implies that $\cH^0(i^*T) = 0$. If $\cH^0(j^!T) \neq 0$, then there is a non-zero map $j_*j^!T \to j_*\cK$. Since $\Hom(i_*i^*T[-1], j_*\cK) = 0$, we get $\Hom(T, j_*\cK) \neq 0$ which is not possible. Therefore, $T \cong [\cO \otimes {V} \xrightarrow{\varphi} E][1]$ for a torsion-free sheaf $E$. Finally, since we have $\Hom([\cO \to 0][1] , T) = 0$, the map $H^0(\varphi)$ is injective.
\end{proof}

The next step is to describe the heart $\cA$ via a torsion pair in $\T$. 
\begin{Lem}\label{lem- case 0 -torsion pair}
\begin{enumerate}
    \item If $T\in \T$ then $T \in \mathcal{P}_{\sigma}(-1, 1]$. %$T = [\cO \otimes V \to E] \in \T$, 
    \item The pair of subcategories $(\mathcal{F}_1, \mathcal{F}_2)$, defined as
\begin{equation*}
    \mathcal{F}_1 = \T \cap \mathcal{P}_{\sigma}(0, 1], \qquad \qquad 
    \mathcal{F}_2 = \T \cap \mathcal{P}_{\sigma}(-1, 0],
\end{equation*}
is a torsion pair on the abelian category $\T$ and the heart $\cA = \mathcal{P}_{\sigma}(0, 1]= \langle \mathcal{F}_2[1], \mathcal{F}_1  \rangle $ is the corresponding tilt.
\end{enumerate}
\end{Lem}
\begin{proof}
The proof is the same as in \cite[Lemma 10.1]{bridgeland:K3-surfaces}, we add it for completeness. For any object $A \in \mathcal{P}_{\sigma}(>1)$, Lemma \ref{lem-case 0-heart} implies that $\cH^i(A) = 0$ for $i \geq 0$, so $\Hom(A, T) = 0$. 
Similarly, if $B \in \mathcal{P}_{\sigma}(\leq -1)$, then $\cH^i(B) = 0$ for $i < 0$, thus $\Hom(T, B) = 0$. This implies $T \in \mathcal{P}_{\sigma}(-1, 1]$ for any $T\in \T$ as claimed in part (a). 

Therefore, any object $T \in \T$ lies in the exact triangle \begin{equation*}
    Q_1 \rightarrow T \rightarrow Q_2
\end{equation*}
with $Q_1 \in \mathcal{P}_{\sigma}(0, 1]$ and $Q_2 \in \mathcal{P}_{\sigma}(-1, 0]$. By Lemma \ref{lem-case 0-heart}, $\cH^i(Q_1) =0$ unless $i = 0, -1$ and $\cH^i(Q_2) = 0$ unless $i = 0, 1$. Then taking cohomology shows that $\cH^{-1}(Q_1) = 0$ and $\cH^1(Q_2) = 0$. This shows that $(\mathcal{F}_1, \mathcal{F}_2)$ is a torsion pair and the heart $\mathcal{A}$ is the corresponding tilt as claimed in part (b).      
\end{proof}

Now we analyze the stability function $Z$. %Take an object $T \in \T$ with $\cl(T) = (r, d, n)$. 
Since $Z(0, 0, 1)$ has zero imaginary part, we get 
\begin{equation*}\label{im}
    \Im[Z(T)] = \alpha \dd(T) - \beta \rd(T)
\end{equation*}
for some $\alpha,\beta \in \mathbb{R}$. Since $j_*\cK \in \cA$ and $\phi_{\sigma}(j_*\cK) <1$, we must have $\alpha > 0$. Then define $$b \coloneqq \frac{\beta}{\alpha}.$$
Thus, up to the $\GL$-action, we may assume 
\begin{equation*}
    \Im[Z(T)] =  \dd(T) \,-\, b\, \rd(T). 
\end{equation*}
\begin{Lem}\label{lem-case 0 - slope stable sheaves}
Consider the torsion pair $(\mathcal{F}_1, \mathcal{F}_2)$ as in Lemma \ref{lem- case 0 -torsion pair}. If $T \in \T$ is $\mu$-stable of positive rank $\rd(T) > 0$, then either $T$ is in $\mathcal{F}_1$ or $\mathcal{F}_2$ depending on whether $\Im[Z(T)] >0$ or $\Im[Z(T)] \leq 0$.    
\end{Lem}
\begin{proof}
We know there is an exact sequence 
\begin{equation*}
    0 \rightarrow Q_1\rightarrow T \rightarrow Q_2 \rightarrow 0
\end{equation*}
in $\T$ where $Q_1 \in \mathcal{F}_1$ and $Q_2 \in \mathcal{F}_2$, so $Q_1 \in \mathcal{A}$ and $Q_2[1] \in \mathcal{A}$. Assume that both $Q_1$ and $Q_2$ are non-zero, otherwise, the claim follows from Lemma \ref{lem-case 0-heart} and Lemma \ref{lem- case 0 -torsion pair}. We know $j^!Q_2 \neq 0$ otherwise, $Q_2[1] = [\cO \otimes V \to 0][1] \in \cA$ for a vector space $V$ which is not possible. Thus Lemma \ref{lem-case 0-heart} implies that $j^!Q_2$ is a torsion-free sheaf, so $\rd(Q_2) \neq 0$. Since $T$ is $\mu$-stable of positive rank, $Q_1$ is also of positive rank. Thus by Lemma \ref{lem- case 0 -torsion pair}, part (b), 
\begin{equation*}
    \frac{1}{\rd(Q_1)} \Im[Z(Q_1)] = \mu(Q_1) -b \geq 0 \qquad \text{and} \qquad \frac{1}{\rd(Q_2)} \Im[Z(Q_2)] = \mu(Q_2) -b \leq 0 
\end{equation*}
which is not possible by $\mu$-stability of $T$. 
%Since $T$ is $\mu$-stable it is possible only if $\mu(Q_1) = \mu(Q_2)=b$, and so $\Im[Z(T)] = 0$. Then we claim $T \in \mathcal{F}_2$. If not, then $T\in \mathcal{F}_1$, so $T \in \cA$ and of phase one. Thus $T$ is $\sigma$-semistable of phase one. Hence Lemma \ref{lem-case 0-heart} implies that $T$ is an extension of $[\cO \to 0]$ and $j_*\cK$ which have rank zero, but we assumed $\rd(T) > 0$, a contradiction.        
\end{proof}
The next step is to determine the real part of the stability function. We can write 
\begin{equation*}
    \Re[Z(T)] = x\,\rd(T) +y\,\dd(T) - z\,\nd(T)
\end{equation*}
for some $x, y, z \in \mathbb{R}$. We know $[\cO \rightarrow 0] \in \cA$ is of phase one, so $z >0$. Up to the $\GL$-action, we may change 
\begin{equation*}
    \Re[Z(T)]\ \  \mapsto\ \  \Re[Z(T)] -y \,\Im[Z(T)] = \rd(T)\, (x+yb) -z \,\nd(T).
\end{equation*}
Since $z>0$, we can also divide it by $z$ and assume 
\begin{equation*}
    \Re[Z(T)] =\rd(T) \,w -\nd(T).
\end{equation*}
for some $w \in \mathbb{R}$. Finally, we claim that $w>\Phi_C(b)$. Indeed, since $\Phi_C$ is upper-semicontinuous, the region
\[
\{(b,w)\in\mathbb{R}^2 \mid w>\Phi_C(b)\}
\]
is open. Hence, by the deformation theory of Bridgeland stability conditions
\cite[Theorem~1.2]{bridgeland:space-of-stability-conditions} or \cite[Theorem~1.2]{bayer:deformation},
it suffices to prove the claim for $b\in\mathbb{Q}$. By construction, for any slope-stable sheaf
$E$ of slope $b$ we have \([\cO\otimes H^0(E)\to E\,][1]\in\cA\), and therefore
\[
\Re \left[Z\bigl([\cO\otimes H^0(E)\to E\,][1]\bigr) \right]
= -\,\rk(E)\, w + h^0(E) \;<\; 0,
\]
so the claim follows from the definition of $\Phi_C$.

\subsection*{Case (II)} Assume $n=0$ and $\phi_{\sigma}(j_*\cK) =1$. 
Take a $\sigma$-stable object $T\in\cA$ that is not isomorphic to $[\cO\to 0]$ or to $j_*\cK$ (for any $x\in C$).
If $\cH^{-1}(T)\neq 0$, it implies that $j^!\mathcal{H}^{-1}(T)\neq 0$ as $[\cO \to 0]\in \mathcal{A}$. Then the injection $\cH^{-1}(T)[1]\hookrightarrow T$ in $\cA$, together with the
nonvanishing $\Hom\!\bigl(j_*\cK,\cH^{-1}(T)[1]\bigr)\neq 0$, implies $\Hom\!\bigl(j_*\cK,T\bigr)\neq 0$,
a contradiction. Therefore, $T\in\T$. By the same argument as in the last part, we conclude that, up to the
$\GL$-action, $\sigma=(\T,Z_\alpha)$ where
\[
Z_\alpha(T)=-\nd(T)-\alpha \dd(T) + i\rd(T) \qquad\text{for some }\alpha>0.
\]
%Z_\alpha(T)=-\dd(T)-\alpha \nd(T) + i\rd(T) \qquad\text{for some }\alpha>0.

\medskip

\subsection*{Case (III)}\label{sub} Suppose $n \geq 1$. Take an object $T \in \cA$. 
Using the vanishing from Lemma~\ref{lem-vanishing of cohomology}, we obtain the following:
\begin{Lem}\label{lem-vanishing i^*T}
We have $\cH^p(i^*T) = 0$ unless $p = 0$, and $\cH^{p}(j^!T) =0$ unless $p = n-1, n$.
\end{Lem} 
\begin{proof}
Lemma~\ref{lem-vanishing of cohomology} implies that $\cH^{p}(j^!T) =0$ unless $p = n-1, n$. So we only need to show the first part of the claim. We know $\cH^{\geq 1}(i^*T) = 0$. Assume $\cH^{<0}(i^*T) \neq 0$. Since $\cH^{\leq -1}(j^!T) = 0$, thus for any $p>0$, we have
    \begin{equation*}
        \Hom([\cO \to 0][p], j_*j^!T[1]) = 0.
    \end{equation*}
    But then taking $\Hom([\cO \to 0][p], -)$ from the exact triangle $T \to i_*i^*T \to j_*j^!T[1]$ implies that $\Hom([\cO \to 0][p], T) \neq 0$, a contradiction. 
  %We know $\cH^{\leq n-2}(j^!T) = 0$. If $\cH^{\geq n+1}(j^!T) \neq 0$, then there is a non-zero map $j_*j^!T \to j_*\cK[-p -n-1]$ for some $p \geq 0$. Since 
   % \begin{equation*}
    %    \Hom([\cO \to 0][-1] , j_*\cK[-p -n-1] ) = 0,
    %\end{equation*}
    %we get $\Hom(T, j_*\cK[-p -n-1]) \neq 0$, a contradiction. 
\end{proof}

Applying a similar argument as in Lemma \ref{lem- case 0 -torsion pair} implies the following:

\begin{Lem}\label{lem-U_0-torsion pair-1}
    For any $F \in \Coh(C)$, we have $j_*F \in \mathcal{P}_{\sigma}(n-1, n+1]$. The pair of subcategories $(\mathcal{F}_1,\, \mathcal{F}_2)$, defined as
\begin{equation*}
    \mathcal{F}_1 = j_*\Coh(C) \cap \mathcal{P}(n, n+1], \qquad \qquad 
    \mathcal{F}_2 = j_*\Coh(C) \cap \mathcal{P}(n-1, n],
\end{equation*}
is a torsion pair on the abelian category $j_*\Coh(C)$ and the heart 
\begin{equation*}\label{eq.heart}
\cA_1 := \langle \mathcal{F}_2[1], \mathcal{F}_1  \rangle [-n]   
\end{equation*}
is the corresponding tilt which is the intersection $\cA \cap j_*\cD(C)$.
\end{Lem}
\begin{proof}
For any object $T\in \mathcal{P}_{\sigma}(>n+1)$, Lemma \ref{lem-vanishing i^*T} implies $\mathcal{H}^{\geq 0}(j^!T)=\mathcal{H}^{\geq -n}(i^*T)=0$. Since $n\geq 1$ it follows that $\mathcal{H}^{\geq 0}(j^*T)=0$. Thus 
    \[
    \Hom(T, j_*F) = \Hom(j^*T, F)=\Hom(\mathcal{H}^0(j^*T)\oplus \mathcal{H}^1(j^*T)[-1], F)=0.
    \]
    Similarly, if $T\in \mathcal{P}_{\sigma}(\leq n-1)$ then Lemma~\ref{lem-vanishing i^*T} implies $\mathcal{H}^{\leq 0}(j^!T)=0$, thus
    \[
    \Hom(j_*F, T)= \Hom(F, j^!T)= \Hom(F, \mathcal{H}^0(j^!T)\oplus \mathcal{H}^{-1}(j^!T)[1])=0. 
    \]
    It follows that $j_*F \in \mathcal{P}_{\sigma}(n-1, n+1]$ as claimed in the first part. 

    Hence, any sheaf $F\in \text{Coh}(C)$ lies in the exact triangle 
    \[
    T_1 \to j_*F \to T_2, 
    \]
    such that $T_1 \in \mathcal{A}[n]$ and $T_2\in \mathcal{A}[n-1]$. From Lemma~\ref{lem-vanishing i^*T}, we have $\mathcal{H}^{\leq -n}(i^*T_2)=0$, so $\mathcal{H}^{-n}(i^*T_1)=0$. It implies that $\mathcal{H}^i(i^*T_1)=\mathcal{H}^i(i^*T_2)=0$ for any $i$. This shows that $(\mathcal{F}_1, \mathcal{F}_2)$ is a torsion pair on $j_*\Coh(C)$ as claimed. 

    %I think the rest is redundant
    %So it is left to show that $\mathcal{A}_1=\cA \cap j_*\cD(C)$. Note that one way inclusion is simple, as if we have that $j_*F\in \mathcal{A}_1$, then $\mathcal{H}^0(j_*F)\in \mathcal{F}_1[-n]$ and $\mathcal{H}^{-1}(j_*F)\in \mathcal{F}_2[-n]$. So we have the distinguished triangle 
    %\[
    %\mathcal{H}^{-1}(j_*F)[1]\to j_*F \to \mathcal{H}^0(j_*F), 
    %\]
    %such that both $\mathcal{H}^{-1}(j_*F)[1], \mathcal{H}^0(j_*F)\in \mathcal{A}\cap j_*\mathcal{D}(C)$. 

    %Now let us take $j_*F\in \mathcal{A}$ and show that $j_*F\in \mathcal{A}_1$. From Lemma~\ref{lem-vanishing i^*T}, we have that $\mathcal{H}^i(F)=0$ unless $i=n-1,n$. Moreover, $\mathcal{H}^{n-1}(j_*F), \mathcal{H}^{n}(j_*F)\in j_*\text{Coh}(C)$. In fact we get 
    %\[
    %\Hom(\mathcal{H}^n(j_*F)[-n], \mathcal{H}^{n-1}(j_*F)[2-n])=\Hom(j^!\mathcal{H}^n(j_*F), j^!\mathcal{H}^{n-1}(j_*F)[2])=0, 
    %\]
    %it follows that $j_*F=\mathcal{H}^{n-1}(j_*F)[1-n]\oplus\mathcal{H}^{n}(j_*F)[-n]$. Since $j_*F\in \mathcal{A}$, we obtain $\mathcal{H}^{n-1}(j_*F)\in \mathcal{F}_2, \mathcal{H}^{n}(j_*F)\in \mathcal{F}_1$. Which finishes the proof that $\mathcal{A}_1=\cA \cap j_*\cD(C)$. 
\end{proof}

Finally, we claim the vanishing $\Hom^{\le 0}([\cO\to 0],\cA_1)=0$. 
Let $j_*F\in\cA_1$. By Lemma~\ref{lem-vanishing i^*T} we have $\cH^{\le -1}(F)=0$. 
Thus,
\[
\Hom^{\le 0}\bigl([\cO\to 0],j_*F\bigr)
=\Hom^{\le -1}(\cO,F)=0,
\]
and the claim follows. Hence, Proposition~\ref{prop 2.2 cp10} implies that, up to $\GL$-action, $\sigma=\gl^{(1)}(\sigma_{\mathcal V},\sigma_g)$ where $\sigma_g=(\cA_1,Z|_{j_*\cD(C)})$ is a stability condition on $\cD(C)$, and $\sigma_{\mathcal V}$ is the trivial stability condition on $\cD(\mathcal V)$.

\begin{proof}[Proof of Theorem \ref{thm-main}]
As explained above, up to the $\GL$-action, any stability condition $\sigma \in \mathrm{Stab}^{\circ}\big(\mathcal{D}(\mathcal{T}_C)\big)$ falls into Case (I), (II), or (III). The first is of Type~B in the theorem, and the latter two are of Type~A; hence the claim follows.
\end{proof}

As a consequence of Theorem \ref{thm-main}, we can describe the complex manifold. 
\begin{Cor}\label{cor complex manifold}
We have
\[
\Stab^{\circ}\big(\cD(\cT_C)\big)=U_A\cup U_B.
\]
where $U_{A}$ and $U_{B}$ are the open loci described in Corollary~\ref{cor-stability-manifold-intro} of the Introduction.
\end{Cor}
\begin{proof}
Pick $\sigma \in U_{A}$, the open locus of stability conditions for which 
\([\cO \to 0],\ [0 \to \cK],\ [0 \to \cO]\) 
are $\sigma$-stable of phases $\phi_{1}, \phi_{2}, \phi_{3}$, respectively. 
Up to rotation, we may assume $\phi_{1} = 1$. 
By Theorem~\ref{thm-main}, $\sigma$ is either a gluing of Type~A or a tilting of Type~B. 
Since $[0 \to \cO]$ is $\sigma$-stable, we deduce that if $\sigma$ arises from tilting $\sigma_{b,w}$ of Type~B, then necessarily $b < 0$. 
By Corollary~\ref{cor-tilting is gluing}, such stability conditions are also of the gluing form of Type~A. 
Hence, every $\sigma \in U_{A}$ is of the form $\sigma = \gl^{(1)}(\sigma_{\mathcal{V}}, \sigma_{g})$ for some $g \in \GL$ satisfying $f(0) < \tfrac{1}{2}$.

The condition $f(0) < \tfrac{1}{2}$ is equivalent to the existence of $k \le 0$ such that 
$\cO[k] \in \Coh^{f(0)}(C)$, which in turn corresponds to $0 < \phi_{3}$. 
The inequalities 
\[
\phi_{3} < \phi_{2} < \phi_{3} + 1
\]
follow from the non-vanishing 
\(\Hom(j_{*}\cK, j_{*}\cO[1]) \neq 0 \neq \Hom(j_{*}\cO, j_{*}\cK)\).  
Proposition~\ref{prop-gluing-type 1} then ensures that $U_{A}$ is precisely the space of triples described in~\eqref{eq.UA}.

Now consider $\sigma \in U_{B}$, the open locus of stability conditions such that 
\([\cO \to 0]\) and \([0 \to \cK]\) are $\sigma$-stable with 
\(\phi_{\sigma}([0 \to \cK]) < \phi_{\sigma}([\cO \to 0])\). 
Up to rotation, we may assume $\phi_{\sigma}([\cO \to 0]) = 1$. 
From the proof of Theorem~\ref{thm-main}, it follows that $\sigma$ belongs to Case~(I), 
so the image of the central charge $Z_{\sigma}$ is not contained in a real line in~$\C$. 
Therefore, the $\GL$-action on $U_{B}$ is free. 
Moreover, Theorem~\ref{thm-tilting} guarantees that the quotient has the claimed description.
%that for any triple of $z_i\in U_{A}$ there exists a unique stability condition $\sigma\in \Stab^{\circ}\big(\cD(\cT_C)\big)$ such that corresponding shifts that belong to the heart of \([\cO\to 0],\ [0\to \cK],\ [0\to\cO]\) have the central charge $z_1, z_2, z_3$ respectively.
\end{proof}

%\begin{Lem}\label{lem-U_0-heart-gl}
%    We have $\cA = \gl (\cA_1, [\cO \to 0])$. 
%\end{Lem}
%\begin{proof}
%    We want to show that the following equality holds 
%\[
%    \mathcal{A}=\{T\in \mathcal{D}(\mathcal{T}_C)~|~ i^*T\in \{\mathbb{C}^{\oplus n}\}_{\geq 0},~ j^!T\in \mathcal{A}_1\}. 
%    \]
%    Note that from Lemma~\ref{lem-vanishing i^*T} one inclusion follows immediately. Let $T\in \mathcal{D}(\mathcal{T}_C)$ be such that it satisfies the right hand side of the equality. Then, we have the following distinguished triangle 
%    \[
%    j_*j^! T \to T \to i_* i^* T, 
%    \]
%    where $i_* i^* T\in \mathcal{A}$ and $j_*j^! T\in \mathcal{A}$ since $\mathcal{A}_1=\cA \cap j_*\cD(C)$. So we have that $T\in \mathcal{A}$ as well.

 %   It remains to check the vanishing $\Hom^{\leq 0}([\cO\to 0], \mathcal{A}_1)=0$. Let $j_*F\in \mathcal{A}$, from Lemma~\ref{lem-vanishing i^*T} we have $\mathcal{H}^{\leq-1}(F)=0$. So the claim follows simply from 
   % \[
   % \Hom^{\leq 0}([\cO\to 0],j_*F)=\Hom^{\leq -1}(\cO, F)=0. 
   % \]
%\end{proof}

%Then Proposition~\ref{prop 2.2 cp10} implies that $\sigma = \gl^{(1)}(\sigma_{\mathcal{V}}, \sigma_g)$ where $\sigma_g =(\cA_1, Z|_{j_*\cD(C)})$ is a stability condition on $\cD(C)$ and $\sigma_{\mathcal{V}}$ is the trivial stability condition on $\mathcal{D}(\mathcal{V})$. Note that since $\Hom(j_*\cO, j_*\cK)=\mathbb{C}$ it follows that $\phi(j_*\cO)<0$, so the inequality $f(0)< \frac{1}{2}$ follows. 

\section{Chamber decomposition and large volume limit}\label{sec chamber dec}
In this section, we describe the wall and chamber decomposition in the two-dimensional slice of
Type~B stability conditions on \(\cD(\cT_C)\) in Theorem \ref{thm-main}. As a consequence, we interpret classical
\(\mu_{\alpha}\)-stability as a large-volume limit along a specified direction and derive a
Bogomolov-type inequality for \(\mu_{\alpha}\)-semistable objects.

We plot the $(b,w)$-plane simultaneously with the image of the projection map
\begin{equation*}
\Pi \colon \mathcal{N}(\mathcal{D}(\T)) \rightarrow \mathbb{R}^2 \;\;\;,\;\;\; \Pi(r,d,n) = \left(\dfrac{d}{r} , \dfrac{n}{r} \right). 
\end{equation*}
Define 
$$
U_C : = \{(b, w): w> \Phi_C(b)\} \subset \mathbb{R}^2.
$$
Note that since $\Phi_C$ is upper semi-continuous, $U_C$ is open. 

\begin{Prop}[\textbf{Wall and chamber structure}]\label{prop. locally finite set of walls}
	Fix $v = (r, d, n)\in \mathcal{N}(\mathcal{D}(\T))$. %such that $\Pi(v) \notin U_{C}$. 
    There exists a set of line segments $\{\ell_i\}_{i \in I}$ in $U_{C}$ (called ``\emph{walls}") which are locally finite and satisfy 
	\begin{itemize*}
	    \item[\emph{(}a\emph{)}] If \(r\neq 0\), then the line containing \(\ell_i\) passes through \(\Pi(v)\).
	    \item[\emph{(}b\emph{)}] If $r=0$ then all $\ell_i$ are parallel of slope $\frac{n}{d}$.
        \item [\emph{(}c\emph{)}] The line segments \(\ell_i\) terminate on the boundary \(\partial U_C\).
	   		\item[\emph{(}d\emph{)}] The $\sigma\_{b,w}$-(semi)stability of any $T \in \cD(\cT_C)$ of class $v$ is unchanged as $(b,w)$ varies within any connected component (called a ``\emph{chamber}") of $U_C \setminus \bigcup_{i \in I}\ell_i$.
		\item[\emph{(}e\emph{)}] For any wall $\ell_i$ there is a map $f\colon T' \to T$ in $\cD(\cT_C)$ such that
\begin{itemize}
\item for any $(b,w) \in \ell_i$, the objects $T', T$ lie in the heart $\cA(b)$,
\item $T$ is $\sigma\_{b,w}$-semistable of class $v$ with $\nu\_{b,w}(T')=\nu\_{b,w}(T)=\,\mathrm{slope}\,(\ell_i)$ constant on the wall $\ell_i$, and
\item $f$ is an injection $T' \hookrightarrow T$ in $\cA(b)$ which strictly destabilises $T$ for $(b,w)$ in one of the two chambers adjacent to the wall $\ell_i$.
\hfill$\square$
\end{itemize} 
	\end{itemize*} 
\end{Prop} 
\begin{proof}
    The argument is identical to the standard proof for tilt stability on the derived category
\(\cD(X)\) of any smooth projective variety \(X\); we omit the repetition and refer to, e.g. \cite[Proposition 4.1]{feyz:thomas-noether-loci} for details.
\end{proof}

As a first application of the wall structure, we obtain a Bogomolov-type inequality for
\(\sigma_{b,w}\)-semistable objects.

\begin{Prop}\label{prop-bg}
Let \(U_{C}^{\mathrm{cvx}} \subset U_{C}\) be an open convex subset. 
If \(T \in \cD(\cT_C)\) with \(\rd(T)\neq 0\) is \(\sigma_{b_0,w_0}\)-semistable for some 
\((b_0,w_0)\in U_{C}^{\mathrm{cvx}}\), then \(\Pi(\cl(T))\notin U_{C}^{\mathrm{cvx}}\).
\end{Prop}

\begin{proof}
Assume, for contradiction, that \(\Pi(\cl(T))\in U_{C}^{\mathrm{cvx}}\). Since \(U_{C}^{\mathrm{cvx}}\) is
convex, the line segment \(\ell\) joining the point \((b_0,w_0)\) to \(\Pi(\cl(T))\) lies entirely inside \(U_{C}^{\mathrm{cvx}}\).
By the structure of walls described in Proposition~\ref{prop. locally finite set of walls}, no wall
separates \((b_0,w_0)\) from any point of \(\ell\); hence \(T\) remains \(\sigma_{b,w}\)-semistable for
all \((b,w)\in\ell\), and, in particular, at \((b_1,w_1):=\Pi(\cl(T))\).
But \(Z_{b_1,w_1}(T)=0\), which contradicts semistability. Therefore,
\(\Pi(\cl(T))\notin U_{C}^{\mathrm{cvx}}\).
\end{proof}

In the next lemma we describe a natural candidate for \(U_{C}^{\mathrm{cvx}}\).

\begin{Lem}\label{lem-classic}
Let \(C\) be a smooth projective curve of genus \(g >1\) with first Clifford index
\(\operatorname{Cliff}_1(C)\ge 2\). Define the piecewise linear function
\[
f(b)=
\begin{cases}
\dfrac{1}{g}\,b + 1 - \dfrac{1}{g}, & 0 < b < 2+ 
\dfrac{2}{g-2}, \\[6pt]
\dfrac{1}{2}\,b, & 2+ 
\dfrac{2}{g-2} \le b < 2g - 4 - \dfrac{2}{g-2}, \\[6pt]
\Bigl(1-\dfrac{1}{g}\Bigr)b + 4 - g - \dfrac{3}{g}, & 2g - 4 - \dfrac{2}{g-2} \le b < 3g - 3, \\[6pt]
b + 1 - g, & 3g - 3 \le b \, .
\end{cases}
\]
Then the region
\[
U_{f} \;:=\; \{\, (b,w)\in\mathbb{R}^2 \mid b>0,\; w> f(b) \,\}
\]
is contained in \(U_C\) and is convex.
\end{Lem}
\begin{proof}
By \cite[Theorem~2.1]{Mercat2002} and \cite[Theorem~4.3]{GrzegorczykTeixidor2009} we have
\(\Phi_C(b)\le f(b)\) for all \(b>0\), hence \(U_f\subset U_C\).
For convexity, observe that \(f\) is piecewise linear with nondecreasing slopes on its
intervals of linearity; thus \(U_f\) is convex. 
\end{proof}
As a direct corollary of Proposition~\ref{prop-bg} and Lemma~\ref{lem-classic}, we obtain the following.
\begin{Cor}\label{Cor-bg}
Any \(\sigma_{b,w}\)-semistable object \(T\in\cD(\cT_C)\) with \(\rd(T)\neq 0\) for some
\((b,w)\in U_f\) satisfies \(\Pi(\cl(T))\notin U_f\).
\end{Cor}
\medskip 

As another application of Proposition~\ref{prop. locally finite set of walls}, we can further investigate
\(\sigma_{b,w}\)-semistable objects for \(b<0\).
For the remainder of this section, we fix a class
\begin{equation*}
    v = (r, d, n) \in \mathcal{N}(\cD(\T)) \qquad \text{with $r, d, n > 0$ }
\end{equation*}
For any $0 \neq \alpha \in \mathbb{R}$, we denote by $\ell_{v}^{\alpha}$ the line of slope $-\frac{1}{\alpha}$ passing through $\Pi(v)$; it is of the equation  
\[
w \;=\; -\frac{1}{\alpha}\!\left(b-\frac{d}{r}\right)+\frac{n}{r}. 
\]

\begin{Lem}\label{lem-negative alpha}
 Assume \(\alpha<0\). If an object \(T=[\cO\otimes V \xrightarrow{\;\varphi\;} E]\in\cT_C\) of class
\(v=(r,d,n)\) is \(\sigma_{b_0,w_0}\)-semistable for some \((b_0,w_0)\in \ell_v^{\alpha}\) with \(b_0<0\), then the morphism
\(\varphi\) is surjective.
\end{Lem}
\begin{proof}
     By the structure of the walls described in Proposition~\ref{prop. locally finite set of walls}, since $b_0<0$ there is no wall separating $(b_0, w_0)$ from any point $(b, w)\in \ell^{\alpha}_v$ where $0< w< w_0$. In particular, it follows that $T$ is $\sigma_{b, w}$-semistable for all $(b, w)\in \ell^{\alpha}_v$ where $0< w\leq  w_0$. 

    Assume that $\varphi$ is not surjective. By the definition of $\cA(b)$, we have a short exact sequence
\[
[\cO\otimes V \to \operatorname{im}(\varphi)] \hookrightarrow T \twoheadrightarrow [0 \to \operatorname{coker}(\varphi)]
\]
in $\cA(b)$, because
\[
b<0=\mu(\cO)\le \mu^{-}\!\bigl(\operatorname{im}(\varphi)\bigr)
\qquad\text{and}\qquad
b<\mu^{-}(E)\le \mu^{-}\!\bigl(\operatorname{coker}(\varphi)\bigr).
\]
Thus $\nu_{b,w}(T)\le \nu_{b,w}\bigl(j_*\operatorname{coker}(\varphi)\bigr)$ for all $(b,w)\in \ell^{\alpha}_v$ with $0<w\le w_0$. This yields
\[
\frac{n-wr}{d-br}
\;\le\;
\frac{-\,w\,\rd \bigl(j_*\operatorname{coker}(\varphi)\bigr)}
{\dd \bigl(j_*\operatorname{coker}(\varphi)\bigr)-b\,\rd \bigl(j_*\operatorname{coker}(\varphi)\bigr)},
\]
which gives a contradiction as $w\to 0$.
\end{proof}
%\color{blue} Do we fix class $v$ in the paragraph above? If yes, then $n>0$. If no, then we define $l^{\alpha}_v$ only for $r,d,n>0$. Also, we need to have that $T\in \mathcal{A}(b)$ at the first place, which implies that $d/r>b$. It is automatic if $d>0$, but if we don't have assumption on $d$, we can have that $T[1]\in \mathcal{A}(b)$. and why is it a lemma, we don't use it later. in the proposition below we need only $r>0, n\neq 0$, no restriction on $d$, apart from "then $T \in \cA(b)$ for any $b < 0$", but we can move $b$ further. ask \color{black}

\begin{Prop}\label{prop-positive alpha}
Assume \(\alpha>0\). An object \(T\in\cD(\cT_C)\) with \(\cl(T)=v\) is
\(\sigma_{b_0,w_0}\)-(semi)stable for some \((b_0,w_0)\in\ell_v^{\alpha}\) with \(b_0 <0\) if and only if
\(T\) is (a shift of) a \(\mu_{\alpha}\)-(semi)stable object of \(\cT_C\).  
\end{Prop}
\begin{proof}
Since $b_0 < 0$ and $\alpha > 0$, it follows that the ray $\ell^{\alpha}_v$ starting at $(b_0, w_0)$ for $b \ll 0$ lies entirely in $U_C$. First, suppose that $T$ is $\sigma_{b_0, w_0}$-(semi)stable. As $T[k]\in \mathcal{A}(b)$ for all $b\ll 0$ and some even shift $k$, we may assume $T \in \cA(b_0)$. The structure of the walls described in Proposition~\ref{prop. locally finite set of walls} implies that $T$ is $\sigma_{b \ll 0, w}$-(semi)stable for $(b, w) \in \ell^{\alpha}_v$. Then $T \in \cT_C$, since the condition $\mu^+(\cH^{-1}(T)) < b \ll 0$ forces $\cH^{-1}(T) = 0$. 

Suppose, for a contradiction, that $T$ is not $\mu_{\alpha}$-(semi)stable, and let 
\begin{equation}\label{t}
    T' \hookrightarrow T \twoheadrightarrow T''
\end{equation}
be a destabilising sequence in $\cT_C$. We may choose $b$ sufficiently small so that $\mu^-(T')>b$, %$b < \mu_{\alpha}^{-}(T) \leq \mu_{\alpha}(T'')$, 
hence \eqref{t} is also a short exact sequence in $\cA(b)$. Then $\nu_{b,w}$-(semi)stability of $T$ implies 
\[
    \frac{\nd(T')r - \bigl(n - \frac{1}{\alpha}(br - d)\bigr)\rd(T')}{\dd(T') - b \rd(T')} 
    < (\leq) 
    \frac{nr - \bigl(n - \frac{1}{\alpha}(br - d)\bigr)r}{d - br}.
\]
After simplification, this becomes 
\[
    0 < (\leq ) \frac{b}{\alpha}\!\left(\frac{\dd(T')}{\rd(T')} - \frac{d}{r}\right)
      + b\!\left(\frac{\nd(T')}{\rd(T')} - \frac{n}{r}\right)
      - \frac{\nd(T')}{\rd(T')}\frac{d}{r}
      + \frac{n}{r}\frac{d}{r}.
\]
For $b \ll 0$, this inequality implies
\[
    \alpha \frac{n}{r} + \frac{d}{r} 
    > (\geq )
    \alpha \frac{\nd(T')}{\rd(T')} + \frac{\dd(T')}{\rd(T')},
\]
and hence $\mu_{\alpha}(T) >(\geq ) \mu_{\alpha}(T')$, a contradiction. 

Conversely, if $T$ is a $\mu_{\alpha}$-(semi)stable object in $\cT_C$ of class $v$, then $T \in \cA(b)$ for any $b < 0$. Suppose $T$ is not $\sigma_{b,w}$-(semi)stable; then there exists a destabilising sequence
\begin{equation}
    T_1 \hookrightarrow T \twoheadrightarrow T_2
\end{equation}
in $\cA(b)$ such that $T_2$ is $\sigma_{b,w}$-semistable when $(b,w) \in \ell^{\alpha}_v$ and $b \ll 0$. Taking cohomology implies that $T_1 \in \cT_C$, and the argument above shows that $\cH^{-1}(T_2) = 0$. Comparing the $\nu_{b,w}$-slopes then contradicts the $\mu_{\alpha}$-(semi)stability of $T$, as established by the above computations.
\end{proof}
Finally, combining Corollary~\ref{Cor-bg} with Proposition~\ref{prop-positive alpha}
yields the following Bogomolov-type inequality for \(\mu_{\alpha}\)-semistable objects.

\begin{Cor}\label{cor-bg-alpha}
Take a $\mu_{\alpha}$-semistable object $T \in \cT_C$ with $\rd(T) \neq 0$, then $\Pi(\cl(T)) \notin U_f$.      
\end{Cor}
%\begin{Lem}
%    If $r=0$ and $\frac{n}{d} \leq 0$, then there is no $\nu_{b,w}$-semistable object $E^{\bullet} \in \cA(b)$ of class $v$ for $b<0$. 
%\end{Lem}

\section{Second type of gluing}\label{sec second type}
In this section, we describe a second open subset \( \widetilde{\Stab}^{\circ}(\mathcal{D}(\T)) \subset \stab(\cD(\mathcal{T}_C))\). Recall that it consists of $\sigma$ such that $[0\to\cO]$ and $[0\to\cK]$ are $\sigma$-stable for all $x\in C$.
Our goal is to prove the following theorem, which shows that all such stability conditions
arise by gluing along a suitable semiorthogonal decomposition. 

\begin{Thm}\label{thm-second open subset}
    If $\sigma\in \widetilde{\Stab}^{\circ}(\mathcal{D}(\T))$ then, up to the $\GL$-action, $\sigma$ is either of the form $\gl^{(1)}(\sigma_{\mathcal{V}}, \sigma_{g})$ for some $g \in \GL$ where $f(0)<\tfrac{1}{2}$ or $\gl^{(2)}( \sigma_g, \sigma_{\mathcal{V}})$ %for some $g \in \GL$ 
    where $f(0)\geq \tfrac{1}{2}$.
\end{Thm}

\subsection*{Geometric stability conditions}
Before proving Theorem~\ref{thm-second open subset}, we first study stability conditions $\sigma$ for which
$j_{*}\cK$ is $\sigma$-stable for every point $x\in C$, without imposing any condition on $[0\to\cO]$. 
By an argument analogous to~\cite[Prop.~2.9]{li-albanese}, we may assume—after the $\GL$-action—that all objects $j_{*}\cK$ are $\sigma$-stable of phase~1. 
The next proposition lists all possible destabilizing sequences for $[\cO\to 0]$.

\begin{Prop}\label{prop destab triangle [O_c -> 0] for geometric}
Let $\sigma$ be a stability condition such that, for every $x\in C$, the object $j_{*}\cK$ is $\sigma$-stable of phase~$1$.
Consider an exact triangle
\begin{equation}\label{eq destab triangle for [O_C -> 0]}
  T_{1} \longrightarrow [\cO\to 0] \longrightarrow T_{2}[1]
\end{equation}
with $T_{1},T_{2}\neq 0$, satisfying $\Hom^{\le 1}(T_{1},T_{2})=0$, where $T_{1}$ is $\sigma$-semistable and all its stable factors are isomorphic, and
\[
\phi_{\sigma}^{+}\!\bigl(T_{2}[1]\bigr) \;\leq\; \phi_{\sigma}(T_{1}).
\]
Then $T_{1},T_{2}\in\T$ and $\cH^{0}\!\bigl(j^{*}T_{1}\bigr)=0$.
\end{Prop}
\begin{proof}
    Applying $j^!$ gives $j^!T_{1}\cong j^!T_{2}$. Moreover, for any $x\in C$ and any $k\in\mathbb{Z}$ we have $\Hom\bigl(j_{*}\cK[k],\,[\cO\to 0]\bigr)=0$. Hence the stable factors of $T_{1}$ (which are all isomorphic) are neither $j_{*}\cK$ nor any of its shifts.

    (1) If $\phi_{\sigma}(T_1)\leq 1$ and $\phi^{+}_{\sigma}(T_2[1])<1$, then $\Hom^{\leq 0}(j_*\cK, T_2[1])=0$, so $\mathcal{H}^{\leq 0}(j^!T_2)=0$. By Lemma~\ref{lem-criterion}(b), it follows that $T_2=0$, a contradiction.

    (2) If $\phi_{\sigma}(T_1)=1$ and $\phi^{+}_{\sigma}(T_2[1])=1$, so $\Hom^{\leq 0}(j_*\cK, T_2)=0$, which means $\mathcal{H}^{\leq -1}(j^!T_2)=0$. Thus, $\mathcal{H}^{\leq -1}(T_1)=\mathcal{H}^{\leq -1}(T_2)=0$ from Lemma~\ref{lem-criterion}(a). We also have $\mathcal{H}^{\ge 0}(j^{*}T_{1})=0$ by Lemma~\ref{lem-vanishing-heart}. Hence, by Lemma~\ref{lem-criterion}(c) we obtain $\mathcal{H}^{\ge 1}(T_{1})=\mathcal{H}^{\ge 1}(T_{2})=0$. In particular, $T_{1},T_{2}\in\mathcal{T}_{C}$ and $\mathcal{H}^{0}(j^{*}T_{1})=0$.

    (3) If $1<\phi_{\sigma}(T_1)<2$, then from Lemma~\ref{lem-vanishing-heart}, we get $\mathcal{H}^{\geq 0}(j^*T_1)=0$ which alongside with Lemma~\ref{lem-criterion}(c) gives $\mathcal{H}^{\geq 1}(T_1)=\mathcal{H}^{\geq 1}(T_2)=0$. Moreover, $\phi^{+}_{\sigma}(T_2[1])\leq \phi_{\sigma}(T_1)<2$, so $\Hom^{\leq 0}(j_*\cK, T_2)=0$ which implies $\mathcal{H}^{\leq -1}(j^!T_2)=0$. Thus $\mathcal{H}^{\leq -1}(T_1)=\mathcal{H}^{\leq -1}(T_2)=0$ from Lemma~\ref{lem-criterion}(a). Therefore, $T_1, T_2\in \mathcal{T}_C$ and $\mathcal{H}^0(j^*T_1)=0$. 

    (4) If $\phi_{\sigma}(T_1)\geq 2$, then from Lemma~\ref{lem-vanishing-heart} it follows that $\mathcal{H}^{\geq -1}(j^*T_1) = 0$ together with Lemma~\ref{lem-criterion}(d) the claim follows. 
\end{proof}
We start with the following useful lemma, which provides with decomposition. 
\begin{Lem}\label{lem-decompos}
    Take $T\in \mathcal{D}(\mathcal{T}_C)$.
    \begin{enumerate}
        \item If $\mathcal{H}^{\leq k}(j^! T)=0$, then $T = i_*V\oplus T'$ for some $V\in \mathcal{D}(\mathcal{V})$ such that $\mathcal{H}^{\geq k+1}(V)=0$ and $\mathcal{H}^{\leq k}(T')=0$.
        \item If $\mathcal{H}^{\geq k}(j^*T)=0$, then $T = i_*' V \oplus T'$ such that $\mathcal{H}^{\leq k}(V)=0$ and $\mathcal{H}^{\geq k+1}(T')=0$. 
    \end{enumerate}
\end{Lem}
\begin{proof}
From $\mathcal{H}^{\leq k}(j^! T)=0$ it follows that $\Hom(\tau^{\leq k} (i_*i^*T), j_*j^!T[1]) = 0$, where $\tau^{\leq k}$ is a natural truncation functor associated to the standard $t$-structure on $\T$. Therefore,   
there is the following commutative diagram
 \begin{equation*}
 	\xymatrix{
     \tau^{\leq k} (i_*i^*T) \ar[r]^{=}\ar[d] & \tau^{\leq k} (i_*i^*T) \ar[r]\ar[d]  & 0\ar[d]  \\
      T \ar[d] \ar[r]& i_*i^*T \ar[r]\ar[d]  & j_*j^!T[1]\ar[d]      \\
     T' \ar[r]& \tau^{\geq k+1} (i_*i^*T) \ar[r]& j_*j^!T[1].
}
  \end{equation*}
On the other hand, from the last raw of the diagram above, we obtain vanishing
  $$
  \Hom(T',\tau^{\leq k} (i_*i^*T)[1]) = \Hom(j_*j^!T, \ \tau^{\leq k} (i_*i^*T)[1]) = 0.
  $$
 Thus $T = \tau^{\leq k} (i_*i^*T) \oplus T' $, which shows the part (a). Similarly part (b) follows from the following commutative diagram
  	   \begin{equation*}
 	\xymatrix{
     j_*j^*T[-1] \ar[r]\ar[d] & \tau^{\leq k} (i'_*i'^!T) \ar[r]\ar[d]  & T'\ar[d]  \\
      j_*j^*T[-1] \ar[d] \ar[r]& i'_*i'^!T \ar[r]\ar[d]  &T\ar[d]      \\
     0 \ar[r]& \tau^{\geq k+1} (i'_*i'^!T) \ar[r]^{=}& \tau^{\geq k+1} (i'_*i'^!T). 
}
  \end{equation*}
\end{proof}

\begin{Lem}\label{lem-criterion} Let $T_1, T_2$ be as in Proposition ~\ref{prop destab triangle [O_c -> 0] for geometric}.
\begin{enumerate}
    \item  If $\mathcal{H}^{\leq k}(j^!T_2)=0$ for some $k<0$ then $\mathcal{H}^{\leq k}(T_1)=\mathcal{H}^{\leq k}(T_2)=0$.
    \item There is $i_0 \leq 0$ such that $\mathcal{H}^{i_0}(j^!T_2) \neq 0$. 
    \item If $\mathcal{H}^{\geq k}(j^*T_1)=0$ 
    for some $k\geq 0$, then $\mathcal{H}^{\geq k+1}(T_1)=\mathcal{H}^{\geq k+1}(T_2)=0$. 
    \item If $\mathcal{H}^{\geq -1}(j^*T_1)=0$, then $T_1 = [\cO \to \cO]$ and $T_2 = [0\to \cO]$. 
\end{enumerate}
\end{Lem}
\begin{proof}
First of all, by adjunction we have
\begin{equation}\label{eq-vanishing-H0-i*T1}
    0 \neq \Hom(T_1, [\cO \to 0]) = \Hom(i^*T_1, \C),
\end{equation}
which implies $\cH^0(i^*T_1)\neq 0$. In other words, the adjunction sends the nonzero map $T_1\to [\cO\to 0]$ from~\eqref{eq destab triangle for [O_C -> 0]} to a nontrivial surjective map $\cH^0(i_*i^*T_1) \twoheadrightarrow [\cO \to 0]$ in $\T$. By \eqref{eq.coho}, we have the short exact sequence in $\cT_C$
$$
\cH^{0}(j_*j^!T_1) \hookrightarrow \cH^0(T_1) \twoheadrightarrow \cH^0(i_*i^{*}T_1) 
$$
which induces the surjection $\cH^0(T_1) \twoheadrightarrow [\cO \to 0]$. Thus, taking cohomology of the exact sequence~\eqref{eq destab triangle for [O_C -> 0]}, implies that $\mathcal{H}^k(T_1)=\mathcal{H}^k(T_2)$ unless $k=0$ and we have the following short exact sequence in $\cT_C$:
\begin{equation}\label{eq-t}
    0 \to \cH^0(T_2) \to \cH^0(T_1) \to [\cO \to 0] \to 0.
\end{equation} 

(a) Suppose there exists $k_0 \leq k < 0$ such that $\cH^{k_0}(T_2)\neq 0$. Then by the decomposition of Lemma \ref{lem-decompos}(a), together with the isomorphism $\cH^{k_0}(T_1) = \cH^{k_0}(T_1)$, we obtain a nonzero morphism $T_1 \to T_2$, which contradicts the assumption that $\Hom^{\leq 1}(T_1, T_2)=0$. 

(b) Suppose, for contradiction, that $\cH^{\leq 0}(j^!T_2) = 0$. By Lemma~\ref{lem-decompos}(a) and part (a) above, we may write 
\[
T_1 \cong i_*V_1 \oplus T'_1, \qquad 
T_2 \cong i_*V_2 \oplus T'_2
\]
where $V_1,V_2$ are finite-dimensional vector spaces and 
\(\cH^{\leq 0}(T'_1) = \cH^{\leq 0}(T'_2)=0\). Since $\mathcal{H}^k(T_1)=\mathcal{H}^k(T_2)$ unless $k=0$, we conclude that 
\[
V_1 \cong V_2 \oplus \C, \qquad T'_1 \cong T'_2.
\]
Hence there always exists a nonzero map $T_1\to T_2$, which contradicts with the assumption $\Hom(T_1, T_2)=0$. 

(c) Applying $j^*$ to the destabilizing sequence~\eqref{eq destab triangle for [O_C -> 0]} gives an exact triangle
\[
\cO \to j^*T_2 \to j^*T_1.
\]
If $k\geq 1$, then $\cH^k(j^*T_2) = \cH^k(j^*T_1)$, so the vanishing $\mathcal{H}^{\geq k}(j^*T_1)=\mathcal{H}^{\geq k}(j^*T_2)=0$ from the assumption, together with Lemma~\ref{lem-decompos}(b), implies the claim as in part (a). %If $\cH^{\geq 0}(j^*T_2) = 0$, then also $\cH^{\geq 0}(j^*T_1)=0$, so again the claim follows. 
It remains to show the claim when $\cH^{\geq 0}(j^*T_1)=0$. Then, from Lemma~\ref{lem-decompos}(b), we get
\[
T_1 \cong i'_*V_1[-1] \oplus T'_1,
\]
for some vector space $V_1$ and $\cH^{\geq 1}(T'_1)=0$. From the previous discussion, we also have $\cH^{\geq 2}(T_2)=0$. We know $\cH^1(T_2)= \cH^1(T_1)= i'_*V_1$. By adjunction we get 
$$\Hom(i'_*V_1[-1], (\tau^{\leq 0}T_2)[1])=\Hom(V_1, i'^!(\tau^{\leq 0}T_2) [2])= 0, $$
so 
\[
T_2 \cong i'_*V_1[-1] \oplus \tau^{\leq 0}T_2
\]
which forces $V_1=0$ as $\Hom(T_1, T_2)=0$. This completes part (c).

\begin{comment}
Hence $d$ must be injective, so 
\[
\cH^1(T_2) \cong [\cO \to 0] \oplus i'_*V_1, 
\]
as $\Hom(i_*'V_1, [\cO \to 0][1])=0$ by adjunction. We then obtain the commutative diagram
\[
\xymatrix{
0 \ar[r]\ar[d] & [\cO \to 0] \ar[r]\ar[d] & [\cO \to 0] \ar[d] \\
\tau^{\leq -1}(T_2[1]) \ar[d]^{=}\ar[r] & T_2[1] \ar[r]\ar[d] & [\cO \to 0]\oplus i'_*V \ar[d] \\
T'_1[1] \ar[r] & T_1[1] \ar[r] & i'_*V
}
\]
If $T'_1\neq 0$, then we obtain a nonzero map 
\[
T_1 \to T'_1 = \tau^{\leq 0}(T_2) \to T_2,
\]
contradiction. Thus $T'_1=0$. But then the semistable factor $T_2[1]\cong [\cO\to 0]\oplus i'_*V$ is impossible.
\end{comment}

\medskip

(d) From part (c) we have $\cH^{\ge 1}(T_1)=\cH^{\ge 1}(T_2)=0$. Moreover, by Lemma~\ref{lem-decompos}(b) we have
\[
T_1=i'_*V_1\oplus T'_1,
\]
where $V_1$ is a vector space and $\cH^{\ge 0}(T'_1)=0$. Note that $V_1\neq 0$ by~\eqref{eq-vanishing-H0-i*T1}. The assumption $\Hom^{\le 1}(T_1,T_2)=0$ implies
\[
0=\Hom^{\le 1}(i'_*V_1,T_2)=\Hom^{\le 1}\bigl(V_1,i'^!T_2\bigr),
\]
which shows $\cH^{\le 1}(i'^!T_2)=0$. Recalling that $i'^!T_2=i^*T_2$ and combining with $\cH^{\ge 1}(T_2)=0$, we obtain $i^*T_2=0$. Thus the short exact sequence \eqref{eq-t} gives $V_1=\Bbb C$ and $\cH^0(T_2)=[0\to \cO]$. Moreover, since $T_2=j_*j^!T_2$, we deduce
\[
T_2=[0\to \cO]\oplus T_2',
\]
where $\cH^{\ge 0}(T_2')=0$. From the exact sequence~\eqref{eq destab triangle for [O_C -> 0]} we get $T_2'=T_1'$, which yields a nonzero morphism $T_1\to T_2$ unless $T_1'=T_2'=0$. Hence the claim follows in this case.

%since otherwise there is a non-zero map $i'_* V_1\to T_2$ which induces a non-zero map $T_1 \to T_2$.  

%Taking cohomology of~\eqref{eq destab triangle for [O_C -> 0]} yields the exact sequence in $\T$
%\[
%0 \to \cH^0(T_2) \to \cH^0(T_1) \to [\cO\to 0] \to 0.
%\]
%Thus, we obtain that
\end{proof}
Similarly to Lemma~\ref{lem-vanishing of cohomology}, we get the following Lemma. 
\begin{Lem}\label{lem-vanishing-heart} Take $\sigma=(\mathcal{A}, Z)$ as in Proposition~\ref{prop destab triangle [O_c -> 0] for geometric}. Let $T[n]\in \mathcal{A}$, then $\mathcal{H}^{\leq n-2}(j^!T)=\mathcal{H}^{\geq n+1}(j^*T)=0$. Moreover, if $T[n]$ is $\sigma$-semistable of phase one whose none of the stable factors is a skyscraper sheaf $j_*\cK$ at a point $x \in C$, then $\mathcal{H}^{\leq n-1}(j^!T)=\mathcal{H}^{\geq n}(j^*T)=0$. 
\end{Lem}
\begin{proof}
    For any $k\geq 0$ and any point $x \in C$, we have 
    \[
    0=\Hom(j_*\mathcal{O}_x[k+1],T[n])=\Hom(\mathcal{O}_x, j^!T[n-k-1])=\Hom(j^!T[n-k-2], \mathcal{O}_x), 
    \]
    \[
0=\Hom(T[n+k+1],j_*\mathcal{O}_x)=\Hom(j^*T[n+k+1], \mathcal{O}_x), 
    \]
    which implies that $\mathcal{H}^{\leq n-2}(j^!T)=\mathcal{H}^{\geq n+1}(j^*T)=0$. The second claim follows similarly. 
\end{proof}

The following lemma provides a complete description of a destabilizing sequence of $[\cO\to 0]$ under the additional assumption that $[0\to\cO]$ is $\sigma$-stable.

\begin{Lem}\label{lem-destab-triangle}
Let $T_1,T_2$ be as in Proposition~\ref{prop destab triangle [O_c -> 0] for geometric}. If $[0\to\cO]$ is $\sigma$-stable, then $T_1=[\cO\to\cO]$ and $T_2=[0\to\cO]$, and $T_1$ is $\sigma$-stable.
\end{Lem}
\begin{proof}
By Proposition~\ref{prop destab triangle [O_c -> 0] for geometric}, we have $T_1,T_2\in\T$ and $\cH^0(j^*T_1)=0$. Write
\[
T_1=[\cO\otimes V \xrightarrow{\ \varphi\ } E].
\]
If $E=0$, then $j^!T_2=E=0$, which is impossible by Lemma~\ref{lem-criterion}(b). Thus $E\neq 0$, the morphism $\varphi$ is surjective with $j^*T_1=\ker(\varphi)[1]$. And we may write
\[
T_2=[\cO\otimes V' \to E],
\]
where $V'$ fits into a short exact sequence of vector spaces
\[
0\to V'\to V\to \Bbb C\to 0.
\]

First assume $\ker(\varphi)\neq 0$. Then there is a morphism
\[
\Hom\bigl(T_1[-1],[0\to \cO\otimes V]\bigr)=\Hom\bigl(j^*T_1[-1],\cO\otimes V\bigr)\neq 0,
\]
given by $\ker(\varphi)\hookrightarrow \cO\otimes V$. %$\ker(\varphi)\hookrightarrow \cO\otimes V \xrightarrow{\ \varphi\ } E$. 
On the other hand, there is a morphism
\begin{equation}\label{nonv}
\Hom\bigl([0\to \cO\otimes V],T_2\bigr)=\Hom\bigl(\cO\otimes V,j^!T_2\bigr)\neq 0,
\end{equation}
induced by $\varphi$. Since $[0\to \cO]$ is $\sigma$-stable, we obtain the inequalities of phases
\[
\phi_\sigma(T_1)-1 \;\le\; \phi_\sigma([0\to \cO]) \;\le\; \phi_\sigma^{+}(T_2).
\]
As $\phi_\sigma^{+}(T_2[1])\le \phi_\sigma(T_1)$ by assumption, we get
\[
\phi_\sigma([0\to \cO])=\phi_\sigma(T_1[-1]).
\]
Moreover, since all $\sigma$-stable factors of $T_1[-1]$ are isomorphic, the nonvanishing in \eqref{nonv} implies that all stable factors of $T_1[-1]$ are isomorphic to $[0\to \cO]$. But then $i^*T_1=0$, contradicting~\eqref{eq-vanishing-H0-i*T1}. Hence $\ker(\varphi)=0$, so $\varphi$ is an isomorphism and $T_1=i'_*V$.

By adjunction we obtain
\[
0=\Hom(T_1,T_2)=\Hom\bigl(V,i'^!T_2\bigr)=\Hom(V,V'),
\]
which forces $V'=0$ and $V=\Bbb C$. Therefore, $T_1=[\cO\to \cO]$ and $T_2=[0\to \cO]$. Finally, the strict $\sigma$-stability of $T_1$ follows from the primitivity of the class $\cl(T_1)$ together with the fact that all its stable factors are isomorphic.
\end{proof}

Now we can proceed to the proof of the main Theorem. 

\begin{proof}[Proof of Theorem \ref{thm-second open subset}]
   Take $\sigma=(\cA,Z)\in \widetilde{\Stab}^{\circ}(\mathcal{D}(\T))$. As before, we may assume that the objects $j_*\cK$ have the same phase for all $x\in C$. We consider two cases, according to the stability of $[\cO\to 0]$.

    \textbf{Case 1: $[\cO \to 0]$ is $\sigma$-stable.} Up to the action of $\GL$, we may assume $\phi_{\sigma}([\cO\to 0])=1$. By adjunction,
\[
\Hom\bigl([\cO\to 0],[0\to \cO][1]\bigr)\neq 0,
\]
hence $0<\phi_{\sigma}([0\to \cO])$. By Theorem~\ref{thm-main}, either $\sigma=\gl^{(1)}(\sigma_{\cV},\sigma_g)$ for some $g\in\GL$, or $\sigma$ is of Type~B (tilting). If $\sigma=\sigma_{b,w}$ is of Type~B, then $0<\phi_{\sigma}([0\to \cO])$ forces $b<0$; by Corollary~\ref{cor-tilting is gluing}, this again implies that $\sigma$ arises from gluing, i.e.\ $\sigma=\gl^{(1)}(\sigma_{\cV},\sigma_g)$ with $f(0)< \tfrac{1}{2}$ as described in Proposition~\ref{prop-gluing-type 1}. This proves the claim in this case.

%Note, that $[0\to \cO]$ is $\mu$-stable in $\T$ as $\cO$. Since $0< \phi([0\to \cO])$ it follows that $[0\to \cO]\in \mathcal{A}$. Therefore, $b<0$ as claimed. 

\textbf{Case 2: $[\cO\to 0]$ is strictly $\sigma$-semistable or $\sigma$-unstable.}
By Lemma~\ref{lem-destab-triangle}, $[\cO\to \cO]$ is $\sigma$-stable and
\begin{equation}\label{order phase}
    \phi_{\sigma}([0\to \cO][1]) \leq \phi_{\sigma}([\cO\to \cO]).
\end{equation}
Up to the action of $\GL$, we may assume $\phi_{\sigma}([\cO\to \cO])=1$, which also gives $\phi_{\sigma}([0\to \cO])\le 0$.
We claim that in this case $\sigma$ comes from gluing, namely $\sigma=\gl^{(2)}(\sigma_g,\sigma_{\cV})$ for some $g = (T, f)\in\GL$ with $f(0) \geq \frac{1}{2}$ as described in Proposition \ref{prop-gluing-type 2}. 

We know $\Hom(j_*\cK,\, j_*\cO[1])=\Bbb C$ and $\Hom(j_*\cO,\, j_*\cK)=\Bbb C$, hence the phases satisfy
\begin{equation}\label{eq inequality of phases}
\phi_{\sigma}([0\to \cO]) \;<\; \phi_{\sigma}(j_*\cK) \;<\; \phi_{\sigma}([0\to \cO]) + 1 \;\le\; 1.
\end{equation}
In particular, $j_*\cK[n]\in\cA$ for some $n\ge 0$.
We now proceed as in Section~\ref{sub}, Case~(III), dividing the argument into steps.

\textbf{Step 1.} We show that for any $T\in\cA$ we have $\cH^k(i'^!T)=0$ unless $k=0$, and $\cH^k(j^*T)=0$ unless $k=-n-1,-n$.

As in Lemma~\ref{lem-vanishing of cohomology}, we obtain $\cH^{\le -n-2}(j^!T)=\cH^{\ge -n+1}(j^*T)=0$. Moreover, we have the vanishing
\begin{equation}\label{van}
    \Hom^{<0}([\cO\to \cO],T)=\cH^{<0}(i'^!T)=0.
\end{equation}
Recall that by Lemma~\ref{lem-simple} there is an exact sequence
\begin{equation}\label{s.e.s}
\cO\otimes i'^!T \longrightarrow j^!T \longrightarrow j^*T.
\end{equation}
Combining \eqref{van} with $\cH^{\le -n-2}(j^!T)=0$, we deduce
$\cH^{\le -n-2}(j^*T)=0$ since $n\ge0$. Hence $\cH^k(j^*T)=0$ unless $k=-n-1,-n$, proving the second part of the claim.

It remains to show the vanishing of $\cH^k(i'^!T)$ for $k > 0$. By Lemma~\ref{lem-decompos}(b) and vanishing $\mathcal{H}^{\geq -n+1}(j^*T)=0$, we can write
\[
T = i'_*V \oplus T' ,
\]
where $\cH^{\le -n+1}(V)=0$ and $\cH^{\ge -n+2}(T')=0$.  
Since $T\in\cA$, we also have
\[
\Hom^{<0}(T,[\cO\to \cO])=0,
\]
which implies $\cH^{>0}(V)=0$. Then:

\begin{enumerate}
    \item[(i)] If $n\ge1$, we have $\cH^{\ge1}(T')=0$, hence $\cH^{\ge1}(i'^!T')=\cH^{\ge1}(i^*T')=0$. Together with $\cH^{>0}(V)=0$, this gives $\cH^{>0}(i'^!T)=0$, as claimed.
    \item[(ii)] If $n=0$, then $V=0$ and hence $\cH^{\ge 2}(T)=0$. Thus it remains to show $\cH^{1}(i'^!T)=0$.  
Combining \eqref{order phase} with \eqref{eq inequality of phases} yields $[0\to \cO][1]\in\cA$, which implies $[\cO\to 0]\in \mathcal{A}$, and therefore
\[
\Hom\bigl(T[1],[\cO\to 0]\bigr)
=\cH^{1}\bigl(i'^!T\bigr)
=0,
\]
as required.
\end{enumerate}

    %If $n\geq 1$, we have that $\mathcal{H}^{>0}(i'^!T)=\mathcal{H}^{>0}(V)=0$ which together with vanishing $\mathcal{H}^{<0}(i'^!T)=0$ shows the claim for $n\geq 1$. 

    %The only case remain to consider is $n=0$, namely when $j_*\cK\in \mathcal{A}$. In this case we get $V=0$ and $\mathcal{H}^k(T)=0$ unless $k=0,1$. 
    
    %From inequality~\ref{eq inequality of phases} we also have $[0\to \cO]\in \mathcal{A}[-1]$, then from the destabilizing triangle for $[\cO\to 0]$ from Proposition~\ref{lem-destab-triangle}, it follows $[\cO\to 0]\in \mathcal{A}$. This gives us the vanishing 
    %\[
    %\Hom(T[1], [\cO\to 0])=\mathcal{H}^1(i^*T)=\mathcal{H}^1(i^!T)=0, 
    %\]
    %which shows the claim. 

\textbf{Step 2.} We claim that for any $F \in \Coh(C)$, we have $j_*F \in \mathcal{P}_{\sigma}(-n-1, -n+1]$. For any object $T\in \mathcal{P}_{\sigma}(>-n+1)$ Step~1 implies that $\mathcal{H}^{\geq 0}(j^*T)=0$, therefore
    \[
    \Hom(T, j_*F) = \Hom(j^*T, F)= 0. %\Hom(\mathcal{H}^0(j^*T)\oplus \mathcal{H}^1(j^*T)[-1], F)=0.
    \]
    Analogously, if $T\in \mathcal{P}_{\sigma}(\leq -n-1)$ then Step~1 implies $\mathcal{H}^{\leq 0}(j^*T)=\mathcal{H}^{\leq 0}(i^*T)=0$, so it follows that 
    $\mathcal{H}^{\leq 0}(j^!T)=0$ by \eqref{s.e.s}, thus
    \[
    \Hom(j_*F, T)= \Hom(F, j^!T)= 0.  %\Hom(F, \mathcal{H}^0(j^!T)\oplus \mathcal{H}^{-1}(j^!T)[1])=0. 
    \]
    This concludes that $j_*F \in \mathcal{P}_{\sigma}(-n-1, -n+1]$ as claimed. 

    Let $(\mathcal{F}_1,\, \mathcal{F}_2)$ be a pair of subcategories defined as
\begin{equation*}
    \mathcal{F}_1 = j_*\Coh(C) \cap \mathcal{P}_{\sigma}(-n, -n+1], \qquad  \qquad 
    \mathcal{F}_2 = j_*\Coh(C) \cap \mathcal{P}_{\sigma}(-n-1, -n]. 
\end{equation*}
Then it is a torsion pair on the abelian category $j_*\Coh(C)$, and $\cA_1 := \langle \mathcal{F}_2[1], \mathcal{F}_1  \rangle [n]$ is the heart of a bounded t-structure on $j_*\cD(C)$. 

 Finally, we show the vanishing $\Hom^{\leq 0}(\mathcal{A}_1, [\cO \to \cO]) = 0$. 
Take $j_*F \in \mathcal{A}_1$. By adjunction, we have
\[
\Hom^{\leq 0}(j_*F, [\cO \to \cO])
 = \Hom^{\leq 0}(F, j^![\cO \to \cO])
 = \Hom^{\leq 0}(F, \cO).
\]
Recall that, by~\eqref{order phase}, we have $\phi_{\sigma}([0 \to \cO]) \le 0$. Hence 
$\Hom^{\leq 0}(j_*F, [0 \to \cO]) = 0$, which implies $\Hom^{\leq 0}(F, \cO) = 0$, as required. 
Therefore, by Proposition \ref{prop 2.2 cp10} and \ref{prop-gluing-type 2}, we conclude that 
$\sigma = \gl^{(2)}(\sigma_g, \sigma_{\cV})$, as claimed. 
\end{proof}
\section{Appendix: Elliptic curves}
In the case where $C$ is an elliptic curve, we completely describe an open subset
\[
\Stab^{\mathrm{Geo}}\big(\mathcal{D}(\T)\big)\subset \Stab\big(\mathcal{D}(\T)\big),
\]
consisting of geometric stability conditions, namely those $\sigma$ for which $j_*\mathcal{O}_x$ is $\sigma$-stable for every $x\in C$, and obtain the following stronger version of Theorem~\ref{thm-main}.

\begin{Thm}\label{thm-elliptic} Let $C$ be an elliptic curve. 
Up to the action of $\GL$, any stability condition $\sigma \in \mathrm{Stab}^{\mathrm{Geo}}\big(\mathcal{D}(\mathcal{T}_C)\big)$ is of one of the following types:  
  \begin{enumerate}
      \item[\textbf{Type A.}] $\sigma$ is the gluing $\gl^{(1)}(\sigma_{\mathcal{V}}, \sigma_g)$ where $g= (T, f) \in \GL$ with $f(0) < \frac{1}{2}$.  
   
   \item[\textbf{Type B.}] The heart of $\sigma$ is $\cA(b) = \langle \mathbb{F}^b[1], \mathbb{T}^b \rangle$ for $b \in \mathbb{R}$ and the stability function is 
  \begin{equation*}
  Z_{b,w} \colon \mathcal{N}(\cD(\mathcal{T})) \rightarrow \mathbb{C} \ ,\;\;\; Z_{b,w}(T) = -\nd(T)+w\rd(T)+i (\dd(T)-b\rd(T))
  \end{equation*}
  where $w > \Phi_C(b)$. 

   \item[\textbf{Type C.}] $\sigma$ is the gluing $\gl^{(2)}(\sigma_g, \sigma_{\mathcal{V}})$ where $g= (T, f) \in \GL$ with $f(0) \geq \frac{1}{2}$.  
   \end{enumerate}
  \end{Thm}
Fix a stability condition $\sigma \in \mathrm{Stab}^{\mathrm{Geo}}\big(\mathcal{D}(\mathcal{T}_C)\big)$. Up to the $\GL$-action, we may assume that $j_*\mathcal{O}_x$ is $\sigma$-stable of phase~$1$ for every point $x \in C$. The proof of Theorem~\ref{thm-elliptic} begins with the following strengthening of Proposition~\ref{prop destab triangle [O_c -> 0] for geometric}, which describes the destabilising sequence of $[\cO\to 0]$.

%, we first show that the destabilizing triangle of $[\cO \to 0]$ is unique in the case when $C$ is an elliptic curve. In particular, we get a sharper version of Lemma~\ref{lem-destab-triangle}.

\begin{Lem}\label{lem-destab-seq-elliptic}
    Consider the destabilizing exact triangle  $T_1 \to [\cO \to 0] \to T_2[1]$ as in Proposition~\ref{prop destab triangle [O_c -> 0] for geometric}. Then
\[
T_1=[\cO\to\cO]
\qquad\text{and}\qquad
T_2=[0\to\cO].
\]
Moreover, both $T_1$ and $T_2$ are $\sigma$-stable.
\end{Lem}
\begin{proof} By Proposition~\ref{prop destab triangle [O_c -> 0] for geometric}, we can write
\[
T_1=[\cO\otimes V \overset{\varphi}{\twoheadrightarrow} E], \qquad T_2 = [\cO\otimes V' \to E]
\]
with $E\neq 0$, $j^*T_1=\ker(\varphi)[1]$ and  $V'$ fits into a short exact sequence of vector spaces
\[
0\to V'\to V\to \Bbb C\to 0.
\]
We first claim that $T_1=[\cO\to\cO]$. To prove the claim, it suffices to show that $\ker(\varphi)=0$, since
\[
\Hom([\cO\to\cO],[\cO\to 0])\cong\mathbb{C},
\]
and the claim then follows. Suppose, for a contradiction, that $\ker(\varphi)\neq 0$.  %the proof of the claim is given in Lemma~\ref{lem-destab-triangle}, note that this part does not depend on the stability of $[0\to \cO]$. 
Since $\Hom([\cO \to 0],T_1)=0$, as the short exact sequence $T_2 \hookrightarrow T_1 \twoheadrightarrow [\cO \to 0]$ in $\cT_C$ is non-split, the map $H^0(\varphi)$ is injective. Hence $H^0(\ker(\varphi))=0$ and %We first claim that $H^0(\varphi)$ is injective. Equivalently, $\Hom([\cO \to 0],T_1)=0$. Indeed, if there is a nonzero map $\iota\colon [\cO \to 0]\to T_1$, then it is injective and, moreover, $T= [\cO \to 0]\oplus \text{coker}(\iota)$ as $[\cO \to 0]$ is an injective object in $\T$, thus the exact sequence splits. As all $\sigma$-stable factors of $T_1$ are isomorphic, we get the contradiction with $E\neq 0$, this concludes that $H^0(\varphi)$ is injective. In particular, it follows that $H^0(\ker(\varphi))=0$. Therefore, we have 
\[
S(T_1)
=
%\left[
%\vcenter{
%  \hbox{
%    $\xymatrix@C=0em@R=2em{
%      \cO \otimes \cone\big(R\Gamma(\ker(\varphi)[1]) \xrightarrow{\pi} V\big)  \ar[d]^{\tilde{ev}} \\
%     \ker(\varphi)[2].
%    }$
%  }
%}
%\right]
%=
\left[
\vcenter{
  \hbox{
    $\xymatrix@C=0em@R=2em{
      \cO \otimes \left(H^1(E)\oplus \faktor{H^0(E)}{V}[1]\right)  \ar[d]^{\tilde{ev}} \\
     \ker(\varphi)[2].
    }$
  }
}
\right].
\]
By assumption, we have the vanishing $\Hom^{\leq 1}(T_1, T_2)=\Hom^{\geq -1}(T_2, S(T_1))=0$. Applying $\Hom(T_2, -)$ to the decomposition~\eqref{rep.1} of $S(T_1)$, we get the following exact sequence 
\[
0=\Hom(T_2, S(T_1)) \to \Hom(T_2, i_*i^*S(T_1)) \to \Hom^1(T_2,j_* \ker(\varphi)[2]) =0. %\to \Hom^1(T_2, S(T_1))=0.
\]
%Since the homological dimension of $\T$ is 2, we obtain that 
This gives
\[
0=\Hom(T_2, i_*i^*S(T_1))=\Hom\left(V', H^1(E)\oplus \faktor{H^0(F)}{V}[1]\right), 
\]
which implies that either $H^1(E)=0$ or $V'=0$. If $H^1(E)=0$, then from Proposition~\ref{prop.hom seqeunce} and $\Hom^1(T_1, T_2)=0$ it follows that $\Hom^1(E,E) = 0$ and so $\Hom(E, E) =0$, a contradiction. If $V'=0$, then $V=\mathbb{C}$. Hence the surjection $\varphi:\cO\twoheadrightarrow E$ with non-zero kernel implies that $\ker(\varphi)$ is a line bundle and $E$ is a zero-dimensional torsion sheaf. Therefore,
\[
0\neq \Hom(\ker(\varphi),E)
=\Hom^1(j^*T_1,E)
=\Hom^1(T_1,T_2)
=0,
\]
a contradiction.

%and we get that  
%\[
%0=\Hom^1(T_1, T_2)=\Hom^1(j^*T_1, E)=\Hom(\ker(\varphi), E). 
%\]
%Since $\ker(\varphi)\neq 0$ and $\ker(\varphi)\hookrightarrow \cO$, we have that $\ker(\varphi)$ is torsion-free sheaf or rank $1$. And, as a consequence, $E$ is a torsion sheaf with $\rk (E)=0$. Thus, we have that 
%\[
%\Hom(\ker(\varphi), E) = \Hom(\cO, E) \neq 0, 
%\]
%this gives a contradiction. It concludes the proof that $T_1 = [\cO \to \cO]$, $T_2=[0\to \cO]$ and $[\cO \to \cO]$ is $\sigma$-stable.

It remains to show that $[0\to \cO]$ is $\sigma$-stable. To prove this, we consider the last piece in the HN filtration of $[0 \to \cO]$, which is a destabilising sequence
\begin{equation}\label{eq-destab-ell}
    T_1' \to [\cO \to 0] \to T_2'[1], 
\end{equation}
with $\Hom^{\leq 1}(T_1', T_2')=0$ such that $T_2'$ is $\sigma$-semistable and all its stable factors are isomorphic. Then we show that $T_1'= [\cO\to \cO]$ and $T_2'=[0\to \cO]$ which directly implies the final statement, as class of $T_2'$ is primitive.  

By the first part of the argument, we already know $[\cO \to \cO]$ is $\sigma$-stable and 
\[
\phi([\cO \to \cO])=\phi^+([\cO \to 0])\geq \phi(T_2'[1]), 
\]
and so
\[
0=\Hom^{\leq 0}([\cO \to \cO], T_2'[1])=\Hom^{\leq 1}(\mathbb{\C}, i'^!T_2'), 
\]
which implies 
\begin{equation}\label{H}
 \mathcal{H}^{\leq 1}(i^*T_2')=0  
\end{equation}
by Lemma~\ref{lem-simple}. This implies that the truncation 
\begin{equation}\label{eq-ell-curve}
    \tau^{\leq 0}T_2' = \tau^{\leq 0}j_*j^!T_2' \to T_2' \to \tau^{\geq 1}T_2'.
\end{equation}
splits as $i^*\cH^1(\tau^{\geq 1}T_2') = 0$ and $$
\Hom^1(j_*j^!\cH^1(\tau^{\geq 1}T_2')[-1] , \tau^{\leq 0}j_*j^!T_2') = 0 = \Hom^1( \tau^{\geq 2}T_2',  \tau^{\leq 0}T_2')
$$  
which gives $\Hom^1(\tau^{\geq 1}T_2', \tau^{\leq 0}T_2'[1]) = 0$ and so splitness of \eqref{eq-ell-curve}. But we know $T_2'$ is $\sigma$-semistable and 
$$
\Hom([\cO \to 0], \tau^{\geq 1}T_2'[1]) = 0
$$
where we again use $i^*\cH^1(\tau^{\geq 1}T_2') = 0$. Hence $\tau^{\geq 1}T_2' = 0$ and $T_2' = \tau^{\leq 0}j_*j^!T_2'$. Since $T_2'$ is $\sigma$-semistable with isomorphic stable factors and there is a nonzero morphism from $[\cO \to 0]$, there are only two possibilities: either $T_2' \cong j_*F$ or $T_2' \cong j_*F'[1]$, for some coherent sheaves $F$ and $F'$. We first show that the second case cannot occur. Suppose otherwise, then $T_1'=[\cO \xrightarrow{\varphi} F'[1]]$ for some morphism $\varphi$ in $\cD(C)$. We have
\begin{equation}\label{eq-0}
    0=\Hom^{1}(T_1',T_2')=\Hom^{1}(T_1',j_*F'[1])=\Hom^2_{\cD(C)}(\cone(\varphi),F'),
\end{equation}
where the last equality follows from Lemma~\ref{lem-simple}. On the other hand, applying $\Hom(-,F')$ to the exact triangle $\cO \xrightarrow{\varphi} F'[1] \to \cone(\varphi)$ yields a surjection
\[
\Hom^2_{\cD(C)}(\cone(\varphi),F') \twoheadrightarrow \Hom^1_{\cD(C)}(F',F') \neq 0.
\]
which contradicts~\eqref{eq-0}. Therefore, $T_2' = j_*F$, and hence $T_1' = [\cO \xrightarrow{\varphi} F]$. Applying the same argument as in the previous case shows that
\begin{equation}\label{eqqq}
\Hom^{\leq 1}(\cone(\varphi),F)=0.    
\end{equation}
On the other hand, since $\Hom([\cO \to 0],T_2')=0$ and the destabilizing sequence~\eqref{eq-destab-ell} is non-split, we also have $\Hom([\cO \to 0],T_1')=0$. It follows that $H^0(\varphi)$ is injective. There are therefore two possibilities: either $\operatorname{im}(\varphi)$ is a torsion sheaf or $\varphi$ is injective. In the first case, we have $\operatorname{cone}(\varphi)\cong L[1]\oplus\widetilde{F}$ for a line bundle $L$ and a sheaf $\widetilde{F}$. However, \eqref{eqqq} implies that $\Hom(L,F)=0$, which is impossible since $F$ contains a nonzero torsion subsheaf. Hence we are in the second case, and $\varphi$ is injective. Applying $\Hom(-,F)$ to the short exact sequence $0\to\cO\xrightarrow{\varphi}F\to\operatorname{coker}(\varphi)\to0$ and using \eqref{eqqq}, we obtain $\Hom^i(F,F)\cong\Hom^i(\cO,F)$ for all $i\in\mathbb{Z}$. This forces $F\cong\cO$, as claimed.

\end{proof}
\begin{proof}[Proof of Theorem \ref{thm-elliptic}]
    Take $\sigma\in \Stab^{\mathrm{Geo}}\big(\mathcal{D}(\T)\big)$. If $[\cO \to 0]$ is $\sigma$-stable, then the claim follows from Theorem~\ref{thm-main}. If $[\cO \to 0]$ is $\sigma$-unstable or strictly $\sigma$-semistable, then from Lemma~\ref{lem-destab-seq-elliptic}, it follows that $[0\to \cO]$ is $\sigma$-stable. Thus the claim follows from Theorem~\ref{thm-second open subset}. 
\end{proof}

\bibliography{mybib}
 \bibliographystyle{halpha}

\end{document}